\documentclass[12pt]{amsart}
\usepackage[utf8]{inputenc}
\usepackage{amsmath}
\usepackage{amssymb}

\newtheorem{theorem}{Theorem}[section]
\newtheorem{defin}[theorem]{Definition}
\newtheorem{lemma}[theorem]{Lemma}
\newtheorem{proposition}[theorem]{Proposition}
\newtheorem{example}[theorem]{Example}
\newtheorem{corollary}[theorem]{Corollary}
\newtheorem{remar}[theorem]{Remark}
\renewenvironment{proof}{Proof:\ \ \ }{\qed}
\newenvironment{remark}{\begin{remar}\rm}{\end{remar}}
\newenvironment{definition}{\begin{defin}\rm}{\end{defin}}
\newcommand{\bfind}[1]{\index{#1}{\bf #1}}
\newcommand{\gloss}[1]{\glossary{#1}{#1}}
\newcommand{\n}{\par\noindent}

\newcommand{\sn}{\par\smallskip\noindent}
\newcommand{\mn}{\par\medskip\noindent}
\newcommand{\bn}{\par\bigskip\noindent}
\newcommand{\pars}{\par\smallskip}
\newcommand{\parm}{\par\medskip}
\newcommand{\parb}{\par\bigskip}

\newcommand{\R}{\mathbb R}
\newcommand{\Q}{\mathbb Q}
\newcommand{\N}{\mathbb N}

\newcommand{\cB}{\mathcal B}
\newcommand{\cC}{\mathcal C}
\newcommand{\cX}{\mathcal X}
\newcommand{\cS}{\mathcal S}
\newcommand{\cN}{\mathcal N}
\newcommand{\cM}{\mathcal M}

\newcommand{\cBpfs}{{\mathcal B}_{\rm pfs}}
\newcommand{\cBiv}{{\mathcal B}_{\rm ps}}
\newcommand{\cBcb}{{\mathcal B}_{\rm ci}}

\newcommand{\fun}{\mbox{\rm f-un}}
\newcommand{\fint}{\mbox{\rm fic}}   
\newcommand{\nint}{\mbox{\rm ci}}   
\renewcommand{\int}{\mbox{\rm ic}}    
\newcommand{\scl}{\mbox{\rm scl}}

\newcommand{\cal}{\mathcal}

\newcommand{\pr}{^{\mbox{\scriptsize\rm pr}}}
\newcommand{\bpr}{^{\mbox{\scriptsize\rm bpr}}}
\newcommand{\tpr}{^{\mbox{\scriptsize\rm tpr}}}
\newcommand{\Fix}{\mbox{\rm Fix}}

\begin{document}
\title[Measuring the strength of completeness]{A generic approach to measuring the strength of completeness/compactness of various types of spaces and ordered structures}
\author[Hanna \'Cmiel, Franz-Viktor and Katarzyna Kuhlmann]{Hanna \'Cmiel, Franz-Viktor Kuhlmann and Katarzyna Kuhlmann}
\address{University of Szczecin, Institute of Mathematics, ul.~Wielkopolska 15, 70-451 Szczecin,
Poland}
\email{hannacmielmath@gmail.com}
\address{University of Szczecin, Institute of Mathematics, ul.~Wielkopolska 15, 70-451 Szczecin,
Poland}
\email{fvk@usz.edu.pl}
\address{University of Szczecin, Institute of Mathematics, ul.~Wielkopolska 15, 70-451 Szczecin,
Poland}
\email{Katarzyna.Kuhlmann@usz.edu.pl}
\thanks{The first two authors are supported by Opus grant 2017/25/B/ST1/01815 from the National Science Centre of
Poland.\\ Corresponding author: Franz-Viktor Kuhlmann}
\date{November 28, 2020}

\begin{abstract}
With a simple generic approach, we develop a classification that encodes and measures the strength of
completeness (or
compactness) properties in various types of spaces and ordered structures. The approach also allows us to encode
notions of functions being contractive in these spaces and structures. As a sample of possible applications we
discuss metric spaces, ultrametric spaces, ordered groups and fields, topological spaces, partially ordered
sets, and lattices. We describe several notions of completeness in these spaces and structures and
determine their respective strengths. In order to illustrate some consequences of the levels of strength, we
give examples of generic fixed point theorems which then can be specialized to theorems in various applications
which work with contracting functions and some completeness property of the underlying space.

Ball spaces are nonempty sets of nonempty subsets of a given set. They are called spherically complete if every
chain of balls has a nonempty intersection. This is all that is needed for the encoding of completeness notions.
We discuss operations on the sets of balls to determine when they lead to larger sets of balls; if so, then the
properties of the so obtained new ball spaces are determined. The operations can lead to increased level of
strength, or to ball spaces of newly constructed structures, such as products. Further,
the general framework makes it possible to transfer concepts and approaches from one application to the other;
as examples we discuss theorems analogous to the Knaster--Tarski Fixed Point Theorem for lattices and theorems
analogous to the Tychonoff Theorem for topological spaces.
\end{abstract}
\maketitle

\newpage
\tableofcontents
\bn
%
%
\section{Introduction}
In view of the notions of completeness of metric spaces, spherical completeness of ultrametric
spaces and compactness of topological spaces, the question arose how these notions can be
``reconciled'', which indicates the search for some ``umbrella'' notion. The question was
triggered in the early 1990s by the appearance of an ultrametric version of Banach's Fixed Point
Theorem (see \cite{[PC]}), which turned out to be a useful tool in valuation theory. An attempt
at finding a generic fixed point theorem for ``metric and order fixed point theory'' was made by
M.~Kostanek and P.~Waszkiewicz in an unpublished paper in the early 2010s. However, the structure
they introduced for this purpose is quite involved. Likewise, it was noticed in private
communications in the late 1990s (inspired by the article \cite{[KU3]}) that ultrametric
fixed point and related theorems appear to have a deeper topological background; but it was
only in 2011 that this observation led to the ideas for the article \cite{[KK1]}, in which ball
spaces were first introduced. They allowed us to extract the essential core of the proofs of
several fixed point theorems and present it in the simplest possible structure. The resulting
``umbrella theorems'' were then used in \cite{[KU3],[KK2],KKP,[KKSo]}) to prove fixed point
theorems in several different settings.

While fixed point theory was the driving force behind this development, the notions we introduced
are fundamental and have a multitude of other aspects and applications. They helped to shape the
approach and results on symmetrically complete ordered fields in \cite{[KKSh]}. Analogues of
basic notions from topology are studied in
\cite{BKK}. The articles \cite{KKub,KKub2} are dealing with problems in ultrametric and ball
spaces that arose when ultrametric spaces were investigated from the particular ball spaces point
of view. In \cite{BCLS} the ball spaces approach is used to prove several principles that are
related to Banach's Fixed Point Theorem but are not themselves fixed point theorems.

The purpose
of the present paper is to systematically develop the abstract theory of ball spaces and to
provide a centerpiece that ties the various applications and directions of research together.
While presenting several new results, its aim is also to present an overview.

A main goal is to show that ball spaces are suitable to encode various completeness notions,
and to measure and compare the strength of these notions. Fixed point theorems will
be used to illustrate the consequences of the level of strength and to show
(with more details than in \cite{[KK1]}) how the umbrella notion makes it
possible to formulate generic fixed point theorems which then can be specialized to theorems in
the various applications.

The inspiration for the minimal structure that allows the encoding of notions
of completeness is taken from ultrametric spaces and their notions of ``ultrametric ball'' and
``spherically complete''. We recall here the basic definitions that were first
introduced in \cite{[KK1]}.

\begin{definition}
 A \bfind{ball space} $(X,\cB)$ consists of a nonempty set $X$ together with a nonempty family
$\cB$ of distinguished nonempty subsets $B$ of $X$ (called {\bf balls}).
\end{definition}

Note that $\cB$, a subset of the power set ${\cal P}(X)$, is partially ordered by inclusion;
we will write $(\cB,\subseteq)$ when we refer to this partially ordered set (in short: poset).

\begin{definition}
A \bfind{nest of balls} in $(X,\cB)$ is a nonempty totally ordered subset of $(\cB,\subseteq)$.
A ball space $(X,\cB)$ is
called \bfind{spherically complete} if every nest of balls has a nonempty intersection.
 \end{definition}

We note that if $(X,\cB)$ is spherically complete and if $\cB'\subseteq\cB$, then also $(X,\cB')$
is spherically complete.

\pars
Beyond the basic notion of ``spherically complete'', we will distinguish various levels of
spherical completeness, which
then provide a tool for measuring the strength of completeness in the spaces and ordered
structures under consideration. On the one hand, we can specify what the intersection of a nest
really is, apart from being nonempty. On the other hand, we can consider intersections of more
general collections of balls than just nests.

A \bfind{directed system of balls} is a nonempty collection of balls such
that the intersection of any two balls in the collection contains a ball
included in the collection. A \bfind{centered system of balls} is a nonempty
collection of balls such that the intersection of any finite number of
balls in the collection is nonempty. Note that every nest is a directed system, and
every directed system is a centered system (but in general, the converses are not true).

\pars
We introduce the following hierarchy of spherical completeness properties:
\sn
{\bf S}$_1$: The intersection of each nest in $(X,\cB)$ is nonempty.
\sn
{\bf S}$_2$: The intersection of each nest in $(X,\cB)$ contains a ball.
\sn
{\bf S}$_3$: The intersection of each nest in $(X,\cB)$ contains maximal balls.
\sn
{\bf S}$_4$: The intersection of each nest in $(X,\cB)$ contains a largest ball.
\sn
{\bf S}$_5$: The intersection of each nest in $(X,\cB)$ is a ball.
\sn
{\bf S}$_i^d$: The same as {\bf S}$_i$, but with ``directed system'' in place of ``nest''.
\sn
{\bf S}$_i^c$: The same as {\bf S}$_i$, but with ``centered system'' in place of ``nest''.
\sn
Note that {\bf S}$_1$ is just the property of being spherically complete. We will use both names, depending on the context.

\pars
The strongest of these properties is {\bf S}$_5^c$; we will abbreviate it as {\bf S}$^*$ as it
will play a central role, enabling us to prove useful results about several important ball
spaces that have this property (it is the
``star'' among the above properties). In Section~\ref{sectpos} we will define an even stronger
(but much more rare) property, namely that arbitrary intersections of balls are again balls.

\parm
We have the following implications:
\mn
\begin{equation}                                 \label{hier}
\begin{array}{ccccccc}
\mbox{\bf S}_1  & \Leftarrow & \mbox{\bf S}_1^d & \Leftarrow & \mbox{\bf S}_1^c &&\\
\Uparrow &  & \Uparrow  &  & \Uparrow \\
\mbox{\bf S}_2  & \Leftarrow & \mbox{\bf S}_2^d & \Leftarrow & \mbox{\bf S}_2^c &&\\
\Uparrow &  & \Uparrow  &  & \Uparrow \\
\mbox{\bf S}_3  & \Leftarrow & \mbox{\bf S}_3^d & \Leftarrow & \mbox{\bf S}_3^c &&\\
\Uparrow &  & \Uparrow  &  & \Uparrow \\
\mbox{\bf S}_4  & \Leftarrow & \mbox{\bf S}_4^d & \Leftarrow & \mbox{\bf S}_4^c &&\\
\Uparrow &  & \Uparrow  &  & \Uparrow \\
\mbox{\bf S}_5  & \Leftarrow & \mbox{\bf S}_5^d & \Leftarrow & \mbox{\bf S}_5^c &=&
\mbox{\bf S}^*
\end{array}
\end{equation}

A question which will be addressed at various points in this paper is under which conditions some
of the implications can be reversed. For instance, it will be shown in Corollary~\ref{S4=S4d}
that {\bf S}$_4$ and {\bf S}$_4^d$ are equivalent.

\parm
In Section~\ref{sectgFPTs} we exemplify the (explicit or implicit) use of spherical completeness and its stronger
versions by presenting generic fixed point theorems for ball spaces. We discuss various ways of encoding the
property of a function of being contractive in the ball space language. We demonstrate the flexibility
of ball spaces, which allows us to taylor them to the specific function under consideration. In connection with
Theorem~\ref{GFPT2} we introduce the idea of associating with every element $x\in X$ a ball $B_x\in\cB$,
leading to the very useful notion of ``{\bf B}$_x$--ball space''.

\pars
The proofs for the generic fixed point theorems will be given in Section~\ref{sectZL}. We use Zorn's Lemma as
the main tool in two different ways: it can be applied to the set of all balls as well as to the set of all
nests, as both are partially ordered by inclusion.

\pars
The properties of hierarchy~(\ref{hier}) will be studied in more detail in Section~\ref{secth}.
We clarify the connection between properties in the hierarchy and properties of posets. Finally
we reveal the strong properties of ball spaces that are closed under various types of nonempty
intersections of balls.

\pars
In Section~\ref{sectappl} we discuss the ways in which ball spaces can be associated with metric spaces,
ultrametric spaces, ordered groups and fields, topological spaces, partially ordered sets, and lattices. In each
case we determine which completeness property is expressed by the spherical completeness of the associated ball
space; an overview is given in the table below. We also study the properties of the associated ball spaces, in
particular which of the properties in the hierarchy~(\ref{hier}) they satisfy.

\bn
\begin{center}
\begin{tabular}{|l|l|l|}
\hline
spaces & balls & completeness \\
& & property \\
\hline
\hline
ultrametric spaces& all closed ultrametric balls & spherically \\
& & complete \\
\hline
metric spaces & metric balls with radii & complete \\
& in suitable sets of & \\
& positive real numbers & \\
\hline
totally ordered sets, & & symmetrically\\
ordered abelian & all intervals $[a,b]$ with $a\leq b$ &  
complete \\
groups and fields &  &  \\
\hline
posets & intervals $[a,\infty)$ & inductively \\
& & ordered \\
\hline
topological spaces & all nonempty closed sets & compact \\
\hline
\hline
metric spaces & Caristi--Kirk balls or & complete \\
& Oettli--Th\'era balls  &  \\
\hline
\end{tabular}
\end{center}

\parm
In this table, the second column indicates a ball space whose spherical completeness is
equivalent to the completeness property stated in the third column. In the case of metric spaces,
the intuitive ball space to consider is that of all closed metric balls. However, the spherical
completeness of this ball space in general is stronger than completeness. See
Section~\ref{sectmet} for details.

\pars
The last entry, the second one for metric spaces, is different from all the other ones. In all other cases
the table has to be read as saying that the completeness property of the given space is equivalent to the spherical
completeness of one single associated ball space containing the indicated balls. But if we work with Caristi--Kirk
balls or Oettli--Th\'era balls, then the completeness of the metric space is  equivalent to the spherical
completeness of a whole variety of Caristi--Kirk ball spaces or Oettli--Th\'era ball spaces that
can be defined on it (see Section~\ref{sectmetCK}). While this may appear impracticable at first glance, it turns
out that these types of balls offer a much better ball spaces approach to metric spaces than the metric balls.

\pars
Not only the specialization of the general framework to particular
applications is important. It is also fruitful to develop the abstract theory of ball spaces,
in particular the behaviour of the various levels of spherical completeness
in the hierarchy (\ref{hier}) under basic operations on ball spaces.

\pars
In Section~\ref{sectsub} we study {\bf S}$^*$ ball spaces. Examples are the compact topological
spaces, where we take the balls to be the nonempty closed sets. Their ball spaces are closed
under arbitrary nonempty intersections of balls, and we make use of the results of
Section~\ref{secth}. We show that {\bf S}$^*$ ball spaces allow the definition of what we call
``spherical closures'' of subsets. They help us to deal with ball space structures induced
on subsets of the set underlying the ball space.

\pars
In Section~\ref{sectuc} we consider set theoretic operations on ball spaces, such as their
closure under finite unions or nonempty intersections of balls, and we study the behaviour of
spherical completeness properties under these operations. We use these preparations to associate
a topology to each ball space and show that it is compact
if and only if the ball space is {\bf S}$_1^c$.

\pars
Products of ball spaces will be studied in Section~\ref{sectTy}. In the paper \cite{BKK}, we
discuss a notion of continuity for functions between ball spaces, as well as quotient spaces and
category theoretical aspects of ball spaces. The products we define here turn out to be the
products in a suitable category of ball spaces.

\pars
Further, the fact that a general framework links various quite different
applications can help to transfer ideas, approaches and results from one to
the other. For instance, the Knaster--Tarski Theorem in the theory of complete
lattices (\cite{T}) presents a useful property of the set of fixed points: they form again a
complete lattice. In Section~\ref{sectanKT}, using the results from Section~\ref{sectsub}, we
prove a ball spaces analogue of the Knaster--Tarski Theorem (Theorem~\ref{KTball}), and an
analogue for topological spaces (Theorem~\ref{KTtop}). A further transfer to other settings,
such as ultrametric spaces, is possible and will be presented in a subsequent paper.

Similarly, in Section~\ref{sectprod} the Tychonoff Theorem from topology is proven for ball
spaces and then transferred to ultrametric spaces. To derive the topological Tychonoff Theorem
from its ball spaces analogue, essential use is made of the results of Section~\ref{sectuc}.

\parm
We hope that we have convinced the reader that the advantage of a general framework is (at least) threefold:
\sn
$\bullet$ \ compare the strength of completenes properties in various spaces and ordered
structures, and transfer concepts and results from one to another,
$\bullet$ \ provide generic proofs of results (such as generic fixed point theorems) which then
can be specialized to various applications, \n
$\bullet$ \ exhibit the underlying principles that are essential for theorems working with some
completeness notion in various spaces and ordered structures.

\bn
%
%
\section{Generic fixed point theorems and the notion of ``contractive function''}              \label{sectgFPTs}
Fixed Point Theorems (FPTs) can be divided into two classes: those dealing
with functions that are in some sense ``contracting'', like
Banach's FPT and its ultrametric variant (cf.\ \cite{[PC]},
\cite{[PR2]}), and those that do not use this property (explicitly or
implicitly), like Brouwer's FPT. In this section, we will be concerned with the
first class.

\pars
Under which conditions do ``contracting'' functions have a fixed point? First of all, we have to
say in which space we work, and we have to specify what we mean by
``contracting''. These specifications will have to be complemented
by a suitable condition on the space, in the sense that it is
``rich'' or ``complete'' enough to contain fixed points for all ``contracting'' functions. Ball
spaces constitute a simple minimal setting in which the necessary
conditions on the function and the space can be formulated.

\pars
We will now give examples of generic FPTs for ball spaces. More such theorems and related
results such as coincidence theorems and so-called attractor theorems are presented in
\cite{[KK1],[KK2],KKP,[KKSo]}. In the present paper we will not discuss the uniqueness of fixed
points; see the cited papers for this aspect. However, an exception will be made in
Theorem~\ref{basic1b}, as this will be used later for an interesting comparison with a
topological fixed point theorem proven in \cite{[SWJ]}.

\pars
{\bf For the remainder of this section, we fix a function $f:X\rightarrow X$.} We abbreviate $f(x)$ by $fx$.
Further, we call a subset $S$ of $X$ \bfind{$f$-closed} if $f(S)\subseteq S$. An
$f$-closed set $S$ will be called \bfind{$f$-contracting} if it satisfies $f(S)\subsetneq S$ unless it is a
singleton. In the search for fixed points, it is a possible strategy to try to find $f$-closed singletons $\{a\}$
because then the condition $f(\{a\})\subseteq \{a\}$ implies that $fa=a$. The significance of this idea is
particularly visible in the case of Caristi--Kirk and Oettli--Th\'era ball spaces discussed in
Section~\ref{sectmetCK}.

\pars
The proofs of the following seven theorems can be found in Section~\ref{sectproofs}.
\begin{theorem}                             \label{basic1a}
Assume that the ball space $(X,\cB)$ is an {\bf S}$_1$ ball space.
\sn
1) \ If every $f$-closed subset of $X$ contains an $f$-contracting
ball, then $f$ has a fixed point in each $f$-closed set.
\sn
2) \ If every $f$-closed subset of $X$ \emph{is} an $f$-contracting ball,
then $f$ has a \emph{unique} fixed point.
\end{theorem}

\pars
We will now give examples showing how some of the stronger notions of spherical completeness can be employed in
general FPTs. In the next theorem, observe how stronger assumptions on the ball space and on $f$ allow us to
only talk about $f$-closed balls instead of $f$-closed subsets.
\begin{theorem}                             \label{basic1b}
Assume that $(X,\cB)$ is an {\bf S}$_5$ ball space and that $f(B)\in\cB$ for every $B\in\cB$.
\sn
1) \ If every $f$-closed ball contains an $f$-contracting ball,
then $f$ has a fixed point in each $f$-closed ball.
\sn
2) \ If every $f$-closed ball is $f$-contracting,
then $f$ has a unique fixed point in each $f$-closed ball.
If in addition $X\in\cB$, then $f$ has a unique fixed point.
\end{theorem}

The next theorem is a variation on the first parts of the previous two theorems.
\begin{theorem}                             \label{basic1c}
Assume that $(X,\cB)$ is an {\bf S}$_2$ ball space. If every ball in $\cB$ contains a
fixed point or a smaller ball, then $f$ has a fixed point in every ball.
\end{theorem}

A condition like ``contains a fixed point or a smaller ($f$-closed) ball'' may appear a little unusual at first.
However, a possible algorithm for finding fixed points should naturally be allowed to stop when it has found one,
so from this point of view the condition is quite natural. We also sometimes use a condition like ``each
$f$-closed ball is a singleton or contains a smaller $f$-closed ball''. This implies ``contains a fixed point or a
smaller $f$-closed ball'' because in an $f$-closed singleton $\{a\}$ the element $a$ must be a fixed point. But
this condition is too strong: as we will see below, there are cases where finding a ball with a fixed point is
easier and more natural than finding a singleton. One example are partially ordered sets where the balls are taken
to be sets of the form $[a,\infty)$. On the other hand, Section~\ref{sectmetCK} shows that there are settings in
which in a natural way we are led to finding $f$-closed singletons (cf.~Proposition~\ref{uhok}).

\pars
The assumptions of these theorems can be slightly relaxed by adapting them  to the given function $f$. Instead of
talking about the intersections of all nests of balls, we need information only about the intersections of nests of
$f$-closed balls. Trivially, if $\emptyset\ne\cB'\subseteq \cB$, then also $(X,\cB')$ is a ball space, and if
$(X,\cB)$ is an {\bf S}$_1$ ball space, then so is $(X,\cB')$. This flexibility of ball spaces appeared already
implicitly in Theorem~\ref{basic1b}
where only $f$-closed balls are used; if nonempty, the subset of all $f$-closed balls is also a ball space, and it
inherits important properties from the (possibly) larger ball space. Tayloring the assumptions on the ball space
to the given function also comes in handy in the following refinement of Theorem~\ref{basic1b}. In its formulation,
the condition ``spherically complete'' does not appear explicitly anymore, but is
implicitly present for the ball space that is chosen in dependence on the function $f$.

\begin{theorem}                             \label{BT'}
Assume that for the given function $f$ there is a ball space $(X,\cB^f)$ such that
\sn
{\bf (B1)} \ each ball in $\cB^f$ is $f$-closed,
\n
{\bf (B2)} \ the intersection of every nest of balls in $\cB^f$ is a singleton or contains a
smaller ball $B\in\cB^f$.
\sn
Then $f$ admits a fixed point in every ball in $\cB^f$.
\end{theorem}

\parm
At first glance, certain conditions of these theorems may appear somewhat unusual. But the reader should
note that their strength lies in the fact that we can freely choose
the ball space. For example, it does not have to be a topology, and in fact, for
essentially all of our applications \emph{it should not be}. This makes
it possible to even choose the balls relative to the given function,
which leads to results like the theorem above.


\pars
When uniqueness of fixed points is not required, then in certain settings (such as ultrametric spaces, see
Section~\ref{sectus}) the condition that a function be ``contracting'' on all
of the space can often be relaxed to the conditions that the function just be ``non-expanding'' everywhere and
``contracting'' on orbits. Again, there is some room for relaxation, and
this is why we will now introduce the following notion.
For each $i\in\N$, $f^i$ will denote the $i$-th iteration of $f$, that is, $f^0x=x$ and
$f^{i+1}x=f(f^ix)$.
\begin{definition}
 The function $f$ is called \bfind{ultimately
contracting on orbits} if there is a function
\begin{equation}                     \label{Bx}
X\ni x\>\mapsto\> B_x\in\cB
\end{equation}
such that for all $x\in X$, the following conditions hold:
\sn
{\bf (NB)} \ $x\in B_x\,$, \sn
{\bf (CO)} \ $B_{fx}\subseteq B_x\,$, and if $x\ne fx$, then $B_{f^ix}\subsetneq B_x$ for some $i\geq 1$.
\sn
If in addition (CO) always holds with $i=1$, then we call $f$ \bfind{contracting on orbits}.
\end{definition}

Note that (NB) and (CO) imply that $f^i x\in B_x$ for all $i\geq 0$.

\pars
The second assertion of our next theorem will show that instead of asking for general spherical
completeness, the scope can be restricted to a particular kind of nests.
\begin{definition}
 A nest $\cN$ of balls
is called an \bfind{$f$-nest} if $\cN=\{B_x\mid x\in S\}$ for some set $S\subseteq X$ that is
closed under~$f$.
\end{definition}

\begin{theorem}                             \label{GFPT2}
Assume that the function $f$ on the ball space $(X,\cB)$ is ultimately contracting on orbits and
that for every $f$-nest $\cN$ in this ball space there is some $z\in \bigcap\cN$ such that
$B_z\subseteq\bigcap\cN$. Then for every $x\in X$, $f$ has a fixed point in $B_x\,$.
\end{theorem}

\pars
The following is the ball spaces analogue of the Ultrametric Banach
Fixed Point Theorem first proved in \cite{[PC]}. We will use the
following condition:
\sn
{\bf (C1)} \ For all $x\in X$, if $y\in B_x$, then $B_y\subseteq B_x\,$.
\sn
%
\begin{theorem}                             \label{GFPT2U}
Assume that the function $f$ on the ball space $(X,\cB)$ is ultimately
contracting on orbits and that condition (C1) is satisfied. If $(X,\cB)$ is an {\bf S}$_1$ ball
space, then for every $x\in X$, $f$ has a fixed point in $B_x\,$.
\end{theorem}

\parm
A particularly elegant version of our approach can be given in the
case of Caristi--Kirk and
Oettli--Th\'era ball spaces (see Theorems~\ref{CKFPT} and~\ref{OTFPT}
in Section~\ref{sectmetCK}). These ball spaces are used in complete
metric spaces. Usually, proofs of fixed point theorems in this setting
work with Cauchy sequences, while the use of metric balls
is inefficient and complicated. For this reason, a ball spaces
approach to metric spaces may seem pointless at
first glance. However, it has turned out that ball spaces made up of
Caristi--Kirk or Oettli--Th\'era balls have a
particularly strong property (cf.\ Proposition~\ref{uhok}), which
makes the ball space approach in this case
exceptionally successful, as demonstrated in Section~\ref{sectmetCK}
and the papers \cite{BCLS,KKP}.

To describe the properties of Caristi--Kirk and Oettli--Th\'era balls,
we introduce the following notions for ball spaces.
\begin{definition}
A ball space $(X,\cB)$ is a {\bf B$_x$--ball space} if there is a
function (\ref{Bx}) such that $\cB= \{B_x\mid x\in X\}$. We call a
B$_x$--ball space $(X,\cB)$
\bfind{normalized} if it satisfies condition (NB), and
\textbf{contractive} if condition (C1) and
the following additional condition are satisfied:
\sn
{\bf (C2)} \ For all $x\in X$, if $B_x$ is not a singleton, then there exists $y\in B_x$ such that
$B_y\subsetneq B_x$.
\sn
A B$_x$--ball space $(X,\cB)$ is \bfind{strongly contractive} if it
satisfies (C1) and:
\sn
{\bf (C2s)} \ For all $x\in X$, if $y\in B_x\setminus\{x\}$, then
$B_y\subsetneq B_x$.
\end{definition}
\sn
Note that condition (C2s) implies (C2) as well as that the function (\ref{Bx}) is a bijection.
In particular, every strongly contractive ball space is contractive. Proposition~\ref{CKOTsc}
will show that all Caristi--Kirk and Oettli--Th\'era ball spaces are strongly contractive
normalized B$_x$--ball spaces. Properties of contractive ball spaces are discussed in
Section~\ref{sectcontrbs}.

It will turn out that condition (NB), while present in many applications, is not always necessary
for our purposes. The next theorem has some similarity with Theorem~\ref{GFPT2}, but it does not
require the B$_x$--ball space to be normalized.
\begin{theorem}                             \label{GFPT3}
If $(X,\cB)$ is a spherically complete contractive B$_x$--ball space and the function $f$ satisfies
\begin{equation}                   \label{fxinBx}
fx\in B_x\quad \mbox{ for all } x\in X\>,
\end{equation}
then it has a fixed point in every ball $B\in\cB$.
\end{theorem}
We note that if $(X,\cB)$ is a strongly contractive B$_x$--ball space and the function $f$
satisfies (\ref{fxinBx}), then it also satisfies (CO) (with $i=1$ for all $x$).

\pars
Interestingly, the exceptional strength of the Caristi--Kirk and Oettli--Th\'era ball spaces is shared by the
ball space made up of the final segments $[a,\infty)$ on partially ordered sets. It would be worthwhile to find
more examples of such strong ball spaces.

\parm
The proofs of the above generic fixed point theorems above are based on Zorn's Lemma. They will
be given in Section~\ref{sectZL} after first investigating the relation between partially ordered
sets and ball spaces. In the present paper we are not interested in avoiding the use of the axiom
of choice, nor is it our task to study its equivalence with certain fixed point theorems. For a
detailed discussion of the case of Caristi--Kirk
and Oettli--Th\'era ball spaces, see Remark~\ref{remax}.

\bn
%
%
\section{Zorn's Lemma in the context of ball spaces}              \label{sectZL}
Consider a poset $(T,<)$. By a \bfind{chain} in $T$ we mean a
\emph{nonempty} totally ordered subset of $T$. An element $a\in T$ is
said to be an \bfind{upper bound} of a subset $S\subseteq T$ if $b\leq a$ for all $b\in S$.
A poset is said to be \bfind{inductively ordered} if every chain has an upper bound.

\pars
Zorn's Lemma states that every inductively ordered poset contains maximal elements. By restricting the assertion
to the set of all elements in the chain and above it, we obtain the following more precise assertion:
\begin{lemma}                         \label{ZL}
In an inductively ordered poset, every chain has an upper bound which is a maximal element in the poset.
\end{lemma}

\begin{corollary}                         \label{ZLcor}
In an inductively ordered poset, every element lies below a maximal element.
\end{corollary}

\begin{definition}                        \label{revord}
We order ball spaces $(X,\cB)$ by reverse inclusion, that is, we set $B_1<B_2$ if $B_1\supsetneq
B_2\,$. In this way we obtain a poset $(\cB,<)$. Now nests of balls in $\cB$ correspond to
chains in the poset. A maximal element in the poset $(\cB,<)$ is a \bfind{minimal ball}, i.e., a
ball that does not contain any smaller ball.
\end{definition}

\mn
%
%
\subsection{The case of {\bf S}$_2$ ball spaces}
\mbox{ }\sn
The proof of the following lemma is straightforward:
\begin{lemma}                           \label{S2ub}
The ball space $(X,\cB)$ is {\bf S}$_2$ if and only if every chain in $(\cB,<)$ has an upper bound.
\end{lemma}

From this fact, one easily deduces the following result.
\begin{proposition}                           \label{S2minb}
In an {\bf S}$_2$ ball space, every ball and therefore also the intersection of every nest contains a minimal ball.
If in addition every ball is either a singleton or contains a smaller ball, then every ball and therefore also the
intersection of every nest contains a singleton ball.
\end{proposition}

In view of Lemma~\ref{S2ub} it is important to note
that every {\bf S}$_1$ ball space $(X,\cB)$ can easily be extended to an
{\bf S}$_2$ ball space by adding all singleton subsets of $X$: we define
\[
\cB_s\>:=\> \cB\cup\{\{a\}\mid a\in X\}\>.
\]
The proof of the following result is straightforward.
\begin{lemma}                                     \label{B_s}
The ball space $(X,\cB_s)$ is {\bf S}$_2$ if and only if $(X,\cB)$ is {\bf S}$_1\,$.
\end{lemma}

However, in many situations the point is exactly to prove that a given ball space admits singleton balls. This is
in particular the case when we work with ball spaces that are adapted to a given function, as in Theorem~\ref{BT'}.
In such cases, instead of applying Zorn's Lemma to chains of balls, one can work with chains of nests instead,
as we will discuss in Section~\ref{sectnob}.

\mn
%
%
\subsection{Posets of nests of balls}                   \label{sectnob}
\mbox{ }\sn
We call a poset \bfind{chain complete} if every chain of elements has a least upper bound (which we also call a
\bfind{supremum}). Note that commonly the
condition ``nonempty'' is dropped from the definition of chains, in which case a chain complete poset must have a
least element. However, for our purposes it is more convenient to only consider chains as nonempty totally ordered
sets.
\begin{lemma}
For every ball space $(X,\cB)$, the set of all nests of balls, ordered by inclusion, is a chain complete poset.
\end{lemma}
\begin{proof}
The union over a chain of nests of balls is again a nest of balls, and it is the smallest nest that contains all
nests in the chain.
\end{proof}

\pars
This shows that in particular every chain of nests that contains a given nest $\cN_0$ has an upper bound. Hence
Zorn's Lemma shows:
\begin{corollary}                         \label{maxnest}
Every nest $\cN_0$ of balls in a ball space is contained in a maximal nest.
\end{corollary}

\mn
%
%
\subsection{The case of contractive B$_x$--ball spaces}        \label{sectcontrbs}
\mbox{ }\sn
In general, a (strongly) contractive B$_x$--ball space $(X,\cB)$ may
not contain balls of the form $\{a\}$ for every $a\in X$, in which
case $\cB\subsetneq\cB_s\,$. Hence we cannot apply Lemma~\ref{B_s} in
order to prove that for such ball spaces, {\bf S}$_1$ and {\bf S}$_2$
are equivalent. However, the following proposition provides a
``sufficient'' amount of singleton balls for this purpose. We also
obtain that these singletons satisfy $B_a=\{a\}$ even if
$(X,\cB)$ is not assumed to be normalized.
\begin{proposition}                                                               \label{uhok}
In a contractive B$_x$--ball space, the intersection of a maximal nest of balls, if nonempty, is
a singleton ball of the form $B_a=\{a\}$.
\end{proposition}
\begin{proof}
Let $\cM$ be a maximal nest of balls and assume that $a\in\bigcap\cM$ for some element $a\in X$.
Since $a\in B$ for every ball $B\in\cM$, we obtain from (C1) that $B_a\subseteq B$ for every $B\in\cM$ and
thus $B_a\subseteq\bigcap\cM$. This means that
$\cM\cup\{B_a\}$ is a nest of balls, so by maximality of $\cM$ we have that $B_a\in\cM$. Consequently, $B_a=
\bigcap\cM$. Suppose that $B_a$ is not a singleton. Then by condition (C2) there is some element $b$
such that $B_b\subsetneq B_a\,$ whence $B_b\notin\cM$. But then $\cM\cup\{B_b\}$ is a nest which strictly
contains $\cM$. This contradiction to the maximality of $\cM$ shows that $B_a$ is a singleton. Since
$a\in\bigcap\cM=B_a$, we must have that $B_a=\{a\}$.
\end{proof}

\pars
Since by Corollary~\ref{maxnest} every nest is contained in a maximal nest, this proposition yields:
\begin{theorem}                                 \label{MTscsc}
\mbox{ }\n
1) A contractive B$_x$--ball space is {\bf S}$_1$ if and only if it is {\bf S}$_2\,$.
\sn
2) In a  contractive B$_x$--ball space which is {\bf S}$_1$
every ball $B_x$ contains a singleton ball of the form $B_a = \{a\}$.
\end{theorem}

\mn
%
%
\subsection{Proofs of the fixed point theorems}                \label{sectproofs}
\mbox{ }\sn
Take a ball space $(X,\cB)$ and a function $f: X\rightarrow X$. By $\cB^f$ we will denote the
collection of all $f$-closed balls in $\cB$, provided there exist any. From
Corollary~\ref{maxnest} we infer that every nest in
$(X,\cB)$ and every nest in $(X,\cB^f)$ is contained in a maximal nest.

Under various conditions on $f$ and on $(X,\cB)$ or $(X,\cB^f)$, we
have to make sure that the intersections of such nests contain a fixed point for $f$. The
proof of the following Lemma is straightforward.
\begin{lemma}                                \label{observe}
1) \ If $S$ is an $f$-closed set, then $ff(S)\subseteq f(S)$ since $f(S)\subseteq S$, hence $f(S)$ is
$f$-closed.
\sn
2) \ The intersection over any collection of $f$-closed sets is again an $f$-closed set.
\end{lemma}

\mn
{\bf Proof of Theorem~\ref{basic1a}:} \
Take an {\bf S}$_1$ ball space $(X,\cB)$. For the proof of part 1) of the theorem, assume that
every $f$-closed subset of $X$ contains an $f$-contracting ball $B$. We have to prove that $f$
has a fixed point in each $f$-closed set $S$.

By assumption, $S$ contains an $f$-contracting ball $B$.
By definition, $B$ is $f$-closed. By Corollary~\ref{maxnest} there exists a maximal nest $\cN$
in the set $\cB^f$ of all $f$-closed balls in $\cB$ which contains the nest $\{B\}$.
Then by part 2) of Lemma~\ref{observe}, $\bigcap\cN$ is an $f$-closed set.
By assumption, it contains an $f$-contracting ball $B'$. Suppose that $B'$ is not a singleton.
Then $B'$ properly contains $f(B')$, which by part 1) of Lemma~\ref{observe} is an $f$-closed
set. Again by assumption, it contains an $f$-contracting and hence $f$-closed
ball $B''$. Since $B''\subseteq f(B')\subsetneq B'\subseteq \bigcap\cN$, we find that $\cN\cup
\{B''\}$ is a larger nest than $\cN$, which contradicts the maximality of $\cN$. This proves
that $B'$ is an $f$-closed singleton
contained in $S$ and thus, $S$ contains a fixed point. This proves part 1) of the theorem.

\pars
In order to prove part 2), assume that every $f$-closed subset of $X$ is an $f$-contracting ball.
We have to prove that $f$ has a unique fixed point.

Take any fixed points $x$ and $y$ of $f$. Then the set $S=\{x,y\}$ is $f$-closed, hence by
assumption it is $f$-contracting. Since $f(S)=S$, it must be a singleton, i.e., $x=y$.
\qed

\mn
{\bf Proof of Theorem~\ref{basic1b}:} \
Assume that $(X,\cB)$ is an {\bf S}$_5$ ball space and that $f(B)\in\cB$ for every $B\in\cB$.
Take an arbitrary $f$-closed ball $B_0\in\cB$.

For the proof of part 1) of the theorem, we have to prove, under the assumption that every
$f$-closed ball contains an $f$-contracting ball, that $B_0$ contains a fixed point.

By Corollary~\ref{maxnest} there exists a maximal nest $\cN$ in $\cB^f$ which contains the nest
$\{B_0\}$. By part 2) of Lemma~\ref{observe},
$\bigcap\cN$ is an $f$-closed set. As $(X,\cB)$ is assumed to be an {\bf S}$_5$ ball space,
$\bigcap\cN$ is also a ball, so $\bigcap\cN\in\cB^f$. Hence by assumption, $\bigcap\cN$ contains
an $f$-contracting ball $B$. If this were not a singleton, then it would contain the
smaller ball $f(B)$, which by part 1) of Lemma~\ref{observe} is $f$-closed.
This would give rise to the nest $\cN\cup\{f(B)\}$ that properly contains
$\cN$, contradicting the maximality of $\cN$. Thus, $\bigcap\cN$ is an $f$-closed singleton
contained in $B_0$ and therefore, $B_0$ contains a fixed point.

\pars
For the proof of part 2) of the theorem, we assume that every $f$-closed ball is $f$-contracting;
now we have to prove that $B_0$ contains a fixed point.

Using transfinite induction, we build a nest $\cN$ consisting of $f$-closed balls as follows.
We set $\cN_0:=\{B_0\}$. Having constructed $\cN_\nu$ for some ordinal $\nu$ with
smallest $f$-closed ball $B_\nu\in\cN_\nu\,$, we set $B_{\nu+1}:=
f(B_\nu)\subseteq B_\nu$ and $\cN_{\nu+1}:=\cN_{\nu} \cup \{B_{\nu+1}\}$. By part 1) of
Lemma~\ref{observe}, also $B_{\nu+1}$ is $f$-closed, and by assumption, it is again a ball.

If $\lambda$ is a limit ordinal and we have constructed $\cN_\nu$ for all
$\nu<\lambda$, we observe that the union over all $\cN_\nu$ is a nest
$\cN'_\lambda\,$. We set $B_\lambda:=\bigcap\cN'_\lambda$ and
$\cN_\lambda:=\cN'_\lambda \cup \{B_\lambda\}$. Since $(X,\cB)$ is an {\bf S}$_5$ ball space,
we know that $B_\lambda\in \cB$, and by part 2) of Lemma~\ref{observe}, $B_\lambda$ is $f$-closed.

The construction becomes stationary when we reach an $f$-closed ball $B_\mu$ that does not properly contain
$f(B_\mu)$. By assumption, $B_\mu$ is $f$-contracting, so this means that $B_\mu\subseteq B_0$ is a singleton
$\{x\}$. As it is $f$-closed, $x$ is a fixed point contained in $B_0\,$.

If $x\ne y\in B_0\,$, then $y\notin B_\mu$ which means that there is some $\nu<\mu$ such that
$y\in B_\nu$, but $y\notin B_{\nu+1}=f(B_\nu)$. This shows that $y$
cannot be a fixed point of $f$. Therefore, $x$ is the unique fixed point of $f$ in $B_0\,$.

\pars
The second assertion of part 2), which states that if every $f$-closed ball is $f$-contracting
and $X\in\cB$, then $f$ has a unique fixed point, is an immediate consequence of the first
assertion of part 2), because $X$ is clearly $f$-closed.
\qed

\mn
{\bf Proof of Theorem~\ref{basic1c}:} \
Assume that $(X,\cB)$ is an {\bf S}$_2$ ball space and that every ball in $\cB$ contains a
fixed point or a smaller ball. We have to prove that $f$ has a fixed point in every ball.

Take any ball $B_0\in\cB$. By Proposition~\ref{S2minb}, $B_0$ contains a minimal ball $B$. As
$B$ cannot contain a smaller ball, it must
contain a fixed point by assumption, which then is also an element of $B_0\,$.
\qed

\mn
{\bf Proof of Theorem~\ref{BT'}:} \
Assume that $\cB^f$ is a ball space of $f$-closed balls and that the intersection of every
nest of balls in $\cB^f$ is a singleton or contains a smaller ball $B\in\cB^f$.
We have to prove that $f$ has a fixed point in every ball $B\in\cB^f$.

Take a maximal nest $\cN$ in $\cB^f$ which contains the nest $\{B\}$. The intersection
$\bigcap\cN$ cannot contain a smaller ball $B'\in\cB^f$ since this would contradict the maximality of $\cN$. Hence
by assumption, $\bigcap\cN$ must be a singleton. As it is also $f$-closed by part 2) of Lemma~\ref{observe} and
contained in $B$, we have proved that $f$ has a fixed point in $B$.
\qed

\pars
For the next two proofs we will use the following fact.
\begin{lemma}                                  \label{maxfnestz}
Take a function $f$ on a ball space $(X,\cB)$.
\sn
1) Every $f$-nest $\cN_0$ in $\cB$ is contained in a maximal $f$-nest.
\sn
2) Assume that $f$ is ultimately contracting on orbits. Assume further that
$\cN$ is a maximal $f$-nest in $\cB$ containing a ball $B_x\,$, and that $z\in
\bigcap\cN$ such that  $B_z\subseteq\bigcap\cN$. Then $z$ is a fixed point of $f$ contained in
$B_x\,$.
\end{lemma}
\begin{proof}
1) The set of all $f$-nests is partially ordered in the following way.
If $\cN_1=\{B_x\mid x\in S_1\}$ and $\cN_2=\{B_x\mid x\in S_2\}$ are $f$-nests
with $S_1$ and $S_2$  closed under $f$, then we define $\cN_1\leq \cN_2$ if $S_1\subseteq S_2\,$.
Then the union over an ascending chain of $f$-nests is again an $f$-nest since the union over
sets that are closed under $f$ is again closed under $f$. Hence by Corollary~\ref{ZLcor}, for
every $f$-nest $\cN_0$ in $\cB$ there is a maximal $f$-nest $\cN$ containing $\cN_0\,$.
\sn
2) If $z\ne fz$ would hold, then by (CO),
$B_{f^iz}\subsetneq B_z\subseteq\bigcap\cN$ for some $i\geq 1$, and the
$f$-nest $\cN\cup\{B_{f^kz}\mid k\in\N\}$ would properly contain $\cN$. But this would
contradict the maximality of $\cN$. Hence, $z\in\bigcap\cN\subseteq B_x\,$ is a fixed point
of~$f$.
\end{proof}

\mn
{\bf Proof of Theorem~\ref{GFPT2}:} \
Take a function $f$ on a ball space $(X,\cB)$ which is ultimately contracting on orbits and
assume that for every $f$-nest $\cN$ in $\cB$ there is some $z\in \bigcap\cN$ such that
$B_z\subseteq\bigcap\cN$. We have to prove that or every $x\in X$, $f$ has a fixed point in
$B_x\,$.

The set $\{B_{f^i x}\mid i\geq 0\}$ is an $f$-nest. Hence by part 1) of Lemma~\ref{maxfnestz}
there is a maximal $f$-nest $\cN$ containing $\{B_{f^i x}\mid i\geq 0\}$. By assumption, there is
some $z\in \bigcap\cN$ such that $B_z\subseteq\bigcap\cN$. By part 2) of Lemma~\ref{maxfnestz},
$z$ is a fixed point of $f$ contained in $B_x\,$.
\qed

\mn
{\bf Proof of Theorem~\ref{GFPT2U}:} \
Assume that the function $f$ on the {\bf S}$_1$ ball space $(X,\cB)$ is ultimately
contracting on orbits and that condition (C1) is satisfied, that is, for all $x\in X$, if
$y\in B_x$, then $B_y\subseteq B_x\,$. We have to prove that or every $x\in X$, $f$ has a fixed
point in $B_x\,$.

By part 1) of Lemma~\ref{maxfnestz} there exists a maximal $f$-nest $\cN$ containing the $f$-nest
$\{B_{f^i x}\mid i\geq 0\}$. Since $(X,\cB)$ is assumed
to be an {\bf S}$_1$ ball space, there is some $z\in \bigcap\cN$. Hence for every
$B_y$ in $\cN$ we have that $z\in B_x$, whence $B_z\subseteq B_y$ by condition (C1).
Consequently, $B_z\subseteq\bigcap\cN$. By part 2) of Lemma~\ref{maxfnestz}, $z$ is a fixed point
of $f$ contained in $B_x\,$.

\mn
{\bf Proof of Theorem~\ref{GFPT3}:} \
Take a spherically complete contractive B$_x$--ball space $(X,\cB)$ and a function
$f:X\rightarrow X$ such that $fx\in B_x$ for all $x\in X$. We have to prove that $f$
has a fixed point in every ball.

By part 2) of Theorem~\ref{MTscsc}, every ball $B_x$ contains a singleton ball of the form
$B_a = \{a\}$. Since $fa\in B_a=\{a\}$, we find that $a$ is a fixed point of $f$ which is
contained in $B_x\,$.

\mn
%
%
\section{Some facts about the hierarchy of ball spaces}            \label{secth}
\subsection{Connection with posets}
\mbox{ }\sn
In this section we will consider properties of the poset $(\cB,<)$ that we derive from a ball
space $(X,\cB)$ via Definition~\ref{revord}, i.e., through ordering $\cB$ by reverse inclusion.

A \bfind{directed system} in a poset is a nonempty subset which
contains an upper bound for any two of its elements. A poset is called
\bfind{directed complete} if every directed system has a least upper bound.
Note that commonly the condition ``nonempty'' is dropped; but for our purposes it is more
convenient to only consider nonempty systems (cf.~our remark in Section~\ref{sectnob}).
As every chain is a directed system, every directed complete poset is chain complete.

The proof of the following observations is straightforward:
\begin{proposition}                    \label{S2S3}
\n
1) A ball space $(X,\cB)$ is {\bf S}$_2$ if and only if $(\cB,<)$ is inductively ordered.
\sn
2) A ball space $(X,\cB)$ is {\bf S}$_2^d$ if and only if
every directed system in $(\cB,<)$ has an upper bound.
\sn
3) A ball space $(X,\cB)$ is {\bf S}$_4$ if and only if $(\cB,<)$ is chain complete.
\sn
4) A ball space $(X,\cB)$ is {\bf S}$_4^d$ if and only if $(\cB,<)$ is directed complete.
\end{proposition}

Let us point out that the intersection of a system of balls may not be
itself a ball, even if it is nonempty (but if it is a ball, then it is
clearly the largest ball contained in all of the balls in the system).
For this reason, in general, the properties {\bf S}$_1\,$,
{\bf S}$_1^d\,$, {\bf S}$_5$ and {\bf S}$_5^d$ cannot be translated into
a corresponding property of $(\cB,<)$. This shows that ball spaces have
more expressive strength than the associated poset structures.

\pars
A proof of the following fact can be found in \cite[p.~33]{C}. See also \cite{M} for generalizations.
\begin{proposition}                  \label{cc=dc}
Every chain complete poset is directed complete.
\end{proposition}
This proposition together with Proposition~\ref{S2S3} yields:
\begin{corollary}                                       \label{S4=S4d}
Every {\bf S}$_4$ ball space is an {\bf S}$_4^d$ ball space.
\end{corollary}

\pars
In the next sections, we will give further criteria for the equivalence of various properties in the hierarchy.

\mn
%
%
\subsection{Singleton balls}        
\mbox{ }\sn
In many applications (e.g. metric spaces with all closed metric balls, ultrametric spaces,
T$_1$ topological spaces) the associated ball spaces
have the property that singleton sets are balls. The following observation is straightforward:
\begin{proposition}                                     \label{ssb}
For a ball space in which all singleton sets are balls, {\bf S}$_1$ is equivalent to {\bf S}$_2\,$,
{\bf S}$_1^d$ is equivalent to {\bf S}$_2^d\,$, and {\bf S}$_1^c$ is equivalent to {\bf S}$_2^c\,$.
\end{proposition}

\mn
%
%
\subsection{Tree-like ball spaces}         \label{sectbsut}
\mbox{ }\sn
We will call a ball space $(X,\cB)$ \bfind{tree-like} if any two balls in $\cB$ with nonempty
intersection are comparable by inclusion. We will see in Section~\ref{sectus} (Proposition~\ref{cust-l})
that the ball spaces associated with classical ultrametric spaces are tree-like.

\begin{proposition}                          \label{thbsut}
In a tree-like ball space, every centered system of balls is a nest. For such a ball space,
{\bf S}$_i\,$, {\bf S}$_i^d$ and {\bf S}$_i^c$ are equivalent, for each $i\in\{1,\ldots,5\}$. If in addition,
in this ball space all singleton sets are balls, then {\bf S}$_1$ is equivalent to {\bf S}$_2^c\,$.
\end{proposition}
\begin{proof}
The first assertion follows from the fact that in a tree-like ball space, every two balls in a centered
system have nonempty intersection and therefore are comparable by inclusion, so the system is a nest. From this,
the second assertion follows immediately. The third assertion follows by way of Proposition~\ref{ssb}.
\end{proof}

\mn
%
%
\subsection{Intersection closed ball spaces}         \label{sectcs}
\mbox{ }\sn
A ball space $(X,{\cal B})$ will be called \bfind{finitely intersection
closed} if ${\cal B}$ is closed under nonempty intersections of any finite collection of balls,
\bfind{chain intersection closed} or \bfind{nest intersection closed} if ${\cal B}$ is closed
under nonempty intersections of nests of balls, and \bfind{intersection
closed} if ${\cal B}$ is closed under nonempty intersections of arbitrary collections of balls.

We will deduce the following result from Proposition~\ref{thbsut}:
\begin{proposition}
Every chain intersection closed tree-like ball space is intersection closed.
\end{proposition}
\begin{proof}
Every collection $\cC$ of balls with nonempty intersection in an arbitrary ball space is a centered system. If
the ball space is tree-like, then by Proposition~\ref{thbsut}, $\cC$ is a nest. If in addition the ball space
is chain intersection closed, then the intersection $\bigcap\cC$ is a ball. Hence under the assumptions of the
proposition, the ball space is intersection closed.
\end{proof}

\parm
The proofs of the following two propositions are straightforward:
\begin{proposition}                             \label{fichier}
Assume that the ball space $(X,\cB)$ is finitely intersection closed.
Then by closing under finite intersections, every centered system of
balls can be expanded to a directed system of balls which has the same
intersection. Hence for a finitely intersection closed ball space, {\bf S}$_i^d$ is equivalent to
{\bf S}$_i^c$, for $1\leq i\leq 5$.
\end{proposition}

\begin{proposition}                           \label{cicSeq}
For chain intersection closed ball spaces, the properties {\bf S}$_1\,$, {\bf S}$_2\,$, {\bf S}$_3\,$,
{\bf S}$_4\,$ and {\bf S}$_5$ are equivalent.
\end{proposition}

As can be expected, the intersection closed ball spaces are the strong\-est when it comes to equivalence of the
properties in the hierarchy.
\begin{theorem}                             \label{sccent}
For an intersection closed ball space, {\bf S}$_1$ is equivalent to {\bf S}$^*$, so
all properties in the hierarchy (\ref{hier}) are equivalent.
\end{theorem}
\begin{proof}
Since $(X,\cB)$ is intersection closed, it is in particular chain intersection closed, hence by
Proposition~\ref{cicSeq}, {\bf S}$_1$ implies {\bf S}$_4\,$. By Corollary~\ref{S4=S4d}, {\bf S}$_4$ implies
{\bf S}$_4^d$. Since $(X,\cB)$ is intersection closed, Proposition~\ref{fichier} shows that {\bf S}$_4^d$
implies {\bf S}$_4^c$. Again since $(X,\cB)$ is intersection closed, the intersection over every directed system
of balls, if nonempty, is a ball; hence {\bf S}$_4^c$ implies {\bf S}$_5^c$. Altogether, we have shown that
{\bf S}$_1$ implies {\bf S}$_5^c$, which shows that all properties in the hierarchy~(\ref{hier})
are equivalent.
\end{proof}

\begin{proposition}                                     \label{S*ic}
Every {\bf S}$^*$ ball space is intersection closed.
\end{proposition}
\begin{proof}
Take any collection of balls with nonempty intersection. Each element in the intersection lies in every ball, so
the collection is a centered system. By assumption, the intersection is again a ball.
\end{proof}

\parm
In a poset, a set $S$ of elements is \bfind{bounded} if and only if it has an upper bound.
A poset is \bfind{bounded complete} if every nonempty bounded set has a least upper bound.
A \bfind{bounded system of balls} is a nonempty collection of balls whose
intersection contains a ball. Note that a bounded system of balls is a
centered system, but the converse is in general not true (not even a nest of balls
is necessarily a bounded system if the ball space is not {\bf S}$_2$).
\pars
The proof of the next lemma is straightforward.
\begin{lemma}
The poset $(\cB,<)$ is bounded complete if and only if the intersection
of every bounded system of balls in $(X,{\cal B})$ contains a largest
ball. In an intersection closed ball space, the intersection of every
bounded system of balls is a ball.
\end{lemma}

\mn
%
%
\subsection{Overview of conditions for equivalences in the hierarchy}     \label{sectceq}
\mbox{ }\sn
The following table will give an overview of conditions for equivalences in the hierarchy
(\ref{hier}) as presented in the previous sections.

\sn
\begin{center}
\begin{tabular}{|l|l|}
\hline
condition on ball spaces & equivalent properties \\
& in the hierarchy \\
\hline
\hline
&\\
no condition & {\bf S}$_4\Leftrightarrow\>${\bf S}$_4^d$\\[4pt]
\hline
&\\
all singletons are balls & {\bf S}$_1\Leftrightarrow\>${\bf S}$_2\,$; \ \
{\bf S}$_1^d\Leftrightarrow\>${\bf S}$_2^d\,$;
\ \ {\bf S}$_1^c\Leftrightarrow\>${\bf S}$_2^c$\\[4pt]
\hline
&\\
tree-like & {\bf S}$_i\Leftrightarrow\>${\bf S}$_i^d\Leftrightarrow\>${\bf S}$_i^c$\\
& for $1\leq i\leq 5$ (each row)\\[4pt]
\hline
&\\
tree-like, all singletons are balls &
{\bf S}$_1\Leftrightarrow\>${\bf S}$_1^d\Leftrightarrow\>${\bf S}$_1^c
\Leftrightarrow\>${\bf S}$_2\Leftrightarrow\>${\bf S}$_2^d
\Leftrightarrow\>${\bf S}$_2^c$\\
& and {\bf S}$_i\Leftrightarrow\>${\bf S}$_i^d\Leftrightarrow\>${\bf S}$_i^c$
for $3\leq i\leq 5$\\[4pt]
\hline
&\\
finitely intersection closed & {\bf S}$_i^d\Leftrightarrow\>${\bf S}$_i^c$ \
for $1\leq i\leq 5$\\[4pt]
\hline
&\\
chain intersection closed & {\bf S}$_1\Leftrightarrow\>${\bf S}$_2
\Leftrightarrow\>${\bf S}$_3\Leftrightarrow\>${\bf S}$_4
\Leftrightarrow\>${\bf S}$_5$ \\
& (first column)\\[4pt]
\hline
&\\
intersection closed & all properties in the hierarchy\\[4pt]
\hline
\end{tabular}
\end{center}

\bn
%
%
\section{Ball spaces and their properties in various applications}    \label{sectappl}
%
In what follows, we will give the interpretation of various levels of spherical completeness in our applications
of ball spaces. At this point, let us define a notion that we will need repeatedly. In a (totally or partially)
ordered set $(S,<)$ a subset $S$ is a \bfind{final segment} if for all $s\in S$, $s<t$ implies $t\in S$;
similarly, $S$ is an \bfind{initial segment} if for all $s\in S$, $s>t$ implies $t\in S$.

%
%
\subsection{Ultrametric spaces}            \label{sectus}
\mbox{ }\sn
For background on ultrametric spaces see \cite{[KU3],[PC],[PR1],[PRCPT],[PR2],[PRlp],[PR3]}.
An \bfind{ultrametric} $u$ on a set $X$ is a function from $X\times X$ to a partially ordered set
$\Gamma$ with smallest element $0$, such that for all $x,y,z\in X$ and all $\gamma\in\Gamma$,
\sn
{\bf (U1)} \ $u(x,y)=0$ \ if and only if \ $x=y$,\n
{\bf (U2)} \ if $u(x,y)\leq\gamma$ and $u(y,z)\leq\gamma$, then
$u(x,z)\leq\gamma$,\n
{\bf (U3)} \ $u(x,y)=u(y,x)$ \ \ \ (symmetry).
\sn
The pair $(X,u)$ is called an \bfind{ultrametric space}. Condition (U2) is the ultrametric triangle law.

We set $uX:=\{u(x,y)\mid x,y\in X\}$ and call it the \bfind{value set of} $(X,u)$. If $uX$ is totally ordered, we
will call $(X,u)$ a \bfind{classical ultrametric space}; in this case, (U2) is equivalent to:
\sn
{\bf (UT)} \ $u(x,z)\leq\max\{u(x,y),u(y,z)\}$.

\pars
We will now introduce three ways of deriving a ball space from an ultrametric space.
A \bfind{closed ultrametric ball} is a set
\[
B_\alpha(x)\>:=\>\{y\in X\mid u(x,y)\leq \alpha\}\>,
\]
where $x\in X$ and $\alpha\in\Gamma$. We obtain the \bfind{ultrametric ball space} $(X,\cB_u)$ from $(X,u)$
by taking $\cB$ to be the set of all such balls $B_\alpha(x)$.

It follows from symmetry and the ultrametric triangle law that every element in a ball is a center, meaning that
\begin{equation}                               \label{balleq}
B_\alpha(x)\>=\>B_\alpha(y)\;\mbox{ if }y\in B_\alpha(x)\>.
\end{equation}
Further,
\begin{equation}                               \label{ballincl}
B_\beta(y)\>\subseteq\> B_\alpha(x)\quad\mbox{ if }\quad y\in B_\alpha(x) \mbox{ and } \beta\leq\alpha\>.
\end{equation}

\pars
A problem with the ball $B_\alpha(x)$ can be that it may not contain any element $y$ such that $u(x,y)=\alpha$;
if it does, it is called \bfind{precise}. It is therefore convenient to work with the precise balls of the form
\[
B(x,y)\>:=\>\{z\in X\mid u(x,z)\leq u(x,y)\}\>,
\]
where $x,y\in X$. We obtain the \bfind{precise ultrametric ball space} $(X,\cB_{[u]})$ from $(X,u)$
by taking $\cB$ to be the set of all such balls $B(x,y)$.

It follows from symmetry and the ultrametric triangle law that
\[
B(x,y)\>=\>B(y,x)
\]
and that
\begin{equation}                            \label{umball}
B(t,z)\subseteq B(x,y) \mbox{ if and only if }
t\in B(x,y) \mbox{ and } u(t,z)\leq u(x,y)\;.
\end{equation}
In particular,
\begin{equation}                        \label{ubconvexprec}
B(t,z)\subseteq B(x,y)\>\mbox{ if }\> t,z\in B(x,y)\>.
\end{equation}
More generally,
\begin{equation}                        \label{ubconvex}
B(t,z)\>\subseteq\>B_\alpha(x)\>\mbox{ if }\> t,z\in B_\alpha(x)\>.
\end{equation}

\pars
Two elements $\gamma$ and $\delta$ of $\Gamma$ are \bfind{comparable} if
$\gamma\leq\delta$ or $\gamma\geq\delta$. Hence if $u(x,y)$ and $u(y,z)$
are comparable, then $B(x,y)\subseteq B(y,z)$ or $B(y,z) \subseteq
B(x,y)$. If $u(y,z)< u(x,y)$, then in addition, $x\notin B(y,z)$. We note:
\begin{equation}                            \label{cb}
u(y,z)\,<\,u(x,y)\>\Longrightarrow\>B(y,z)\,\subsetneq\, B(x,y)\;.
\end{equation}
In classical ultrametric spaces every two values $\alpha,\beta$ are comparable. Hence in this case one can derive
from (\ref{balleq}) and (\ref{ballincl}) that every two ultrametric balls with nonempty intersection are
comparable by inclusion.

\parm
From (\ref{ballincl}), we derive:
\begin{proposition}                           \label{cust-l}
In a classical ultrametric space $(X,u)$, any two balls with nonempty intersection are comparable by inclusion.
Hence $(X,\cB_{[u]})$ and $(X,\cB_u)$ are tree-like ball spaces.
\end{proposition}

\parb
We define $(X,u)$ to be \bfind{spherically complete} if its ultrametric ball space $(X,\cB_u)$ is
spherically complete, i.e., an {\bf S}$_1$ ball space. For this definition, it actually makes no
difference whether we work with $\cB_u$ or $\cB_{[u]}\,$:
\begin{proposition}
The classical ultrametric ball space $(X,\cB_u)$ is spherically complete if and only if the
precise ultrametric ball space $(X,\cB_{[u]})$ is.
\end{proposition}
\begin{proof}
Since $\cB_{[u]}\subseteq\cB_u\,$, the implication ``$\Rightarrow$'' is clear. Now take a nest $\cN$ of balls in
$\cB_u\,$. We may assume that it does not contain a smallest ball since otherwise this ball equals the
intersection over the nest, which consequently is nonempty. Further, there is a coinitial subnest
$(B_{\alpha_\nu}(x_{\nu}))_{\nu<\kappa}$ such that $\kappa$ is an infinite limit ordinal and $\mu<\nu<\kappa$
implies that $B_{\alpha_\nu}(x_{\nu})\subsetneq B_{\alpha_\mu}(x_{\mu})$. It follows that this subnest has the same
intersection as $\cN$.

For every $\nu<\kappa$, also $\nu+1<\kappa$ and thus $B_{\alpha_{\nu+1}}(x_{\nu+1})\subsetneq
B_{\alpha_\nu}(x_{\nu})$. Hence there is $y_{\nu+1} \in B_{\alpha_\nu}(x_{\nu})\setminus
B_{\alpha_{\nu+1}}(x_{\nu+1})$. It follows that
\[
u(x_{\nu+1},y_{\nu+1})\> >\>\alpha_{\nu+1}\>,
\]
and from (\ref{ballincl})
we obtain that
\[
B_{\alpha_{\nu+1}}(x_{\nu+1})\>\subseteq\> B_{u(x_{\nu+1},y_{\nu+1})}(x_{\nu+1}) \>=\> B(x_{\nu+1},y_{\nu+1})\>.
\]
Since $x_{\nu+1},y_{\nu+1}\in B_{\alpha_\nu}(x_{\nu})$, we know from (\ref{ubconvex}) that
\[
B(x_{\nu+1},y_{\nu+1})\>\subseteq\> B_{\alpha_\nu}(x_{\nu})\>.
\]
It follows that
\[
\bigcap\cN\>=\>\bigcap_{\nu<\kappa} B_{\alpha_\nu}(x_{\nu})\>=\>\bigcap_{\nu<\kappa} B(x_{\nu+1},y_{\nu+1})\>.
\]
Consequently, if $\cB_{[u]}$ is {\bf S}$_1\,$, then this intersection is nonempty and we have proved that also
$\cB_u$ is {\bf S}$_1\,$.
\end{proof}

\parm
Since $uX$ contains the smallest element $0:=u(x,x)$, $\cB_u$ contains
all singletons $\{x\}=B_0(x)$. Therefore, each ultrametric ball space is
already {\bf S}$_2$ once it is {\bf S}$_1\,$. The same is true for the precise ultrametric ball space
$(X,\cB_{[u]})$ in place of $(X,\cB_u)$. However, these ball spaces will in general not be {\bf S}$_3\,$,
{\bf S}$_4$ or {\bf S}$_5$ because even if an intersection of a nest is nonempty, it will not necessarily be a
ball of the form $B_\alpha(x)$ or $B(x,y)$, respectively.

In a classical ultrametric space, every two balls are comparable by inclusion once they have nonempty intersection.
Therefore, every centered system is already a nest of balls. This shows:
\begin{proposition}                \label{cussc=s2c}
A classical ultrametric space $(X,u)$ is spherically complete if and only if the ball space
$(X,\cB_u)$ (or equivalently, $(X,\cB_{[u]})$) is an {\bf S}$_2^c$ ball space.
\end{proposition}

If $(X,u)$ is a classical ultrametric space, then we can obtain stronger completeness properties if we work with
a larger set of ultrametric balls. Given $x\in X$ and an initial segment $S\ne \emptyset$ of $uX$, we define:
\[
B_S(x)\>=\>\{y\in X\mid u(x,y)\in S\}\>.
\]
Setting
\[
\cB_{u+}\>:=\> \{B_S(x)\mid x\in X\mbox{ and $S$ a nonempty initial segment of $uX$}\}\>,
 \]
we obtain what we will call the \bfind{full ultrametric ball space} $(X,\cB_{u+})$. Note that
$X=B_{uX}(x)\in\cB_{u+}$. We leave it to the reader to prove:
\begin{equation}                              \label{capcup}
B_S(x)\>=\>\bigcup_{\alpha\in S} B_\alpha(x)\>\subseteq\> \bigcap_{\beta\geq S} B_\beta(x)
\end{equation}
where $\beta\geq S$ means that $\beta\geq\gamma$ for all $\gamma\in S$, and the intersection over an empty index
set is taken to be $X$. We note that the inclusion on the right hand side is proper if and only if $S$ has no
largest element but admits a supremum $\alpha$ in $uX$ and there is $y\in X$ such that $\alpha=u(x,y)$. Indeed,
if $S=\{\beta\mid\beta<\alpha\}$, then $B_S(x)$ is the \bfind{open ultrametric ball}
\[
B^\circ_\alpha(x)\>:=\>\{y\in X\mid u(x,y)< \alpha\}\>,
\]
which is a proper subset of $B_\alpha(x)=\bigcap_{\beta\geq S} B_\beta(x)$ if and only if $B_\alpha(x)$ is precise.

\pars
We have that
\[
\cB_{[u]}\>\subseteq\> \cB_u\>\subseteq\> \cB_{u+}
\]
where the second inclusion holds because $B_\alpha (x)=B_S(x)$ for the initial segment $S=[0,\alpha]$ of $uX$.
We have an easy generalization of (\ref{ubconvex}):
\begin{equation}                        \label{ubconvexf}
\mbox{if }B\in \cB_{u+} \mbox{ and }  t,z\in B\, , \mbox{ then }  B(t,z)\>\subseteq\>B\>.
\end{equation}

\pars
The following results are proven in \cite{KKub2}:
\begin{theorem}                                      \label{ultcic}
Let $(X,u)$ be a classical ultrametric space. Then the following assertions hold.
\sn
1) \ The intersection over every nest of balls in $(X,\cB_{u+})$ is equal to
the intersection over a nest of balls in $(X,\cB_u)$ and therefore, $(X,\cB_{u+})$ is chain intersection closed.
\sn
2) \ The ball space $(X,\cB_{u+})$ is an {\bf S}$_1$ ball space if and only if $(X,\cB_u)$ is.
\sn
3) \ The ball space $(X,\cB_{u+})$ is tree-like and intersection closed.
If $(X,\cB_u)$ is an {\bf S}$_1$ ball space, then $(X,\cB_{u+})$ is an {\bf S}$^*$ ball space.
\end{theorem}

\parm
By \cite[Theorem~1.2]{KKub}, assertions 1) and 2) of Theorem~\ref{ultcic} also hold for all ultrametric spaces
$(X,u)$ with countable
narrow value sets $uX$; the condition \bfind{narrow} means that all sets of mutually incomparable elements in $uX$
are finite. On the other hand, it is shown in \cite{KKub} that the condition ``narrow'' cannot be dropped in this
case. It is however an open question whether the condition ``countable'' can be dropped.

\parb
A large number of ultrametric fixed point and coincidence point theorems have been proven by S.~Prie{\ss}-Crampe
and P.~Ribenboim (see e.g.\ \cite{[PC],[PR1],[PRCPT],[PR2],[PR3]}). Using ball spaces, some of them have been
reproven and new ones have been proven in \cite{[KK1],[KK2],[KKSo]}.

\mn
%
%
\subsection{Metric spaces with metric balls}    \label{sectmet}
\mbox{ }\sn
In metric spaces $(X,d)$ we can consider the closed metric balls
\[
B_\alpha (x)\>:=\>\{y\in X\mid d(x,y)\leq \alpha\}
\]
for $x\in X$ and $\alpha\in\R^{\geq 0}:=\{r\in\R\mid r\geq 0\}$. We set
\[
{\cal B}_d\>:=\>\{B_\alpha (x)\mid x\in X\,,\,\alpha\in\R^{\geq 0}\}\>.
\]
The following theorem will be deduced from Theorem~\ref{scmetr2} below:

\begin{theorem}                             \label{scmetr1}
If the ball space $(X,{\cal B}_d)$ is spherically complete, then
$(X,d)$ is complete.
\end{theorem}
The converse is not true. Consider a rational function field $k(x)$
together with the $x$-adic valuation $v_x\,$. Choose an extension of
$v_x$ to a valuation $v$ of the algebraic closure $K_0$ of $k(x)$. Then the value group is $\Q$.
An ultrametric in the sense of Section~\ref{sectus} is obtained by setting, for instance,
\[
u(a,b)\>:=\>e^{-v(a-b)}\>.
\]
Take $(K,u)$ to be the completion of $(K_0,u)$. It can be shown that the balls
\[
B_{\alpha_i} \left(\sum_{j=1}^{i-1} x^{-\frac 1 j}\right)  \;\;\mbox{ with } \alpha_i \>=\> e^{\frac 1 i}
\qquad (2\leq i\in\N)
\]
have empty intersection in $K$. Hence $(K,u)$ is not spherically complete,
that is, the ultrametric ball space induced by $u$ on $K$ is not
spherically complete. But this ultrametric is a complete metric.

Note that from Theorem~\ref{cts} below it follows that the ball space $(X,\cB_d)$ is spherically
complete if every closed metric ball in $(X,d)$ is compact under the topology induced by $d$, as
the closed metric balls are closed in this topology.

\parm
In order to characterize complete metric spaces by spherical completeness, we have to choose smaller induced ball
spaces. For any subset $S$ of the set $\R^{>0}$ of positive real numbers, we define:
\[
{\cal B}_S\>:=\>\{B_r (x)\mid x\in X\,,\,r\in S\}\>.
\]



\begin{theorem}                             \label{scmetr2}
The following assertions are equivalent:
\sn
a) \ $(X,d)$ is complete,
\sn
b) \ the ball space $(X,\cB_S)$ is spherically complete for \emph{some}
$S\subset\R^{>0}$ which admits $0$ as its only accumulation point,
\sn
c) \ the ball space $(X,\cB_S)$ is spherically complete for \emph{every}
$S\subset\R^{>0}$ which admits $0$ as its only accumulation point.
\end{theorem}
\begin{proof}
a) $\Rightarrow$ c):
Assume that $(X,d)$ is complete and take a set $S\subset\R^{>0}$ which admits $0$ as its only
accumulation point. This implies that $S$ is discretely ordered, hence every infinite descending chain in $S$
with a maximal element can be indexed by the natural numbers.

Take any nest $\cN$ of closed metric balls in $\cB_S$. If
the nest contains a smallest ball, then its intersection is nonempty; so we assume that it does not.
If $B\in\cN$, then $\cN_B:= \{B'\in\cN\mid B'\subseteq B\}$ is a nest of balls with $\bigcap\cN= \bigcap\cN_B\,$;
therefore, we may assume from the start that $\cN$ contains a largest ball. Then the radii of the balls in $\cN$
form an infinite descending chain in $S$ with a maximal element, and 0 is their unique accumulation point. Hence we
can write $\cN=\{B_{r_i}(x_i)\mid i\in \N\}$ with $r_j<r_i$ for $i<j$, and with $\lim_{i\rightarrow\infty} r_i=0$.

For every $i\in\N$ and all $j\geq i$, the element $x_j$ lies in $B_{r_i} (x_i)$ and therefore
satisfies $d(x_i,x_j)\leq r_i\,$. This shows that $(x_i)_{i\in\N}$ is a
Cauchy sequence. Since $(X,d)$ is complete, it has a limit $x$ in $X$.
We have that $d(x_i,x) \leq r_i\,$, so $x$ lies in every ball $B_{r_i}
(x_i)$. This proves that the nest has nonempty intersection.

\sn
c) $\Rightarrow$ b): \ Trivial.

\sn
b) $\Rightarrow$ a): \
Assume that $(X,\cB_S)$ is spherically complete. Take any Cauchy sequence $(x_n)_{n\in {\N}}$ in $X$.
By our assumptions on $S$, we can choose a sequence $(s_i)_{i\in \N}$ in $\{s\in S
\mid s<s_0\}$ such that $0<2s_{i+1}\leq s_i$. Now we will use induction on $i\in\N$ to choose an increasing
sequence $(n_i)_{i\in\N}$ of natural numbers such that the balls $B_i:=B_{s_i}(x_{n_i})$ form a nest.

Since $(x_n)_{n\in {\N}}$ is a Cauchy sequence, we have that
there is $n_1$ such that $d(x_n,x_m)< s_2$ for all $n,m > n_1\,$. Once we have chosen $n_{i-1}\,$,
we choose $n_i>n_{i-1}$ such that $d(x_n,x_m)< s_{i+1}$ for all $n,m \geq n_i\,$. We show that the so obtained
balls $B_i$ form a nest. Take $i\in \N$ and $x \in B_{i+1}=B_{s_{i+1}}(x_{n_{i+1}})$.
This means that $d(x_{n_{i+1}},x)\leq s_{i+1}$. Since $n_i,n_{i+1} \geq n_i$, we have that $d(x_{n_i},x_{n_{i+1}})
< s_{i+1}$. We compute:
\begin{eqnarray}
d(x_{n_i},x)&\leq & d(x_{n_i},x_{n_{i+1}})+ d(x_{n_{i+1}},x)\nonumber \\
&\leq & s_{i+1} + s_{i+1} = 2s_{i+1} \leq s_i\>.\nonumber
\end{eqnarray}
Thus $x\in B_i$ and hence $B_{i+1} \subseteq B_i$ for all $i\in \N$.
The intersection of this nest $(B_i)_{i\in \N}$ contains some $y$, by our assumption.
We have that  $y \in B_i$ for all $i\in \N$, which means that $d(x_{n_i},y)\leq s_i$. Since
\[
\lim_{i \to \infty}s_i=0,
\]
we obtain that
\[
\lim_{i \to \infty}x_{n_i}= y,
\]
which proves that $(X,d)$ is a complete metric space.
\end{proof}

\mn
Proof of Theorem~\ref{scmetr1}: \ Assume that $(X,{\cal B}_d)$ is
spherically complete. Then so is $(X,{\cal B}')$ for every nonempty $\cB'\subset {\cal B}_d\,$. Taking $\cB'=\cB_S$
with $S$ as in Theorem~\ref{scmetr2}, we obtain that $(X,d)$ is complete. \qed

\pars
\begin{remark}
Theorems~\ref{scmetr1} and~\ref{scmetr2} remain true if instead of the closed metric balls the open metric balls
\[
B_\alpha (x)\>:=\>\{y\in X\mid d(x,y)< \alpha\}
\]
are used for the metric ball space.
\end{remark}

\mn
%
%
\subsection{Metric spaces with Caristi--Kirk balls and Oettli--Th\'era balls}
\label{sectmetCK}
Consider a metric space $(X,d)$. A function $\varphi:X\to\R$ is \bfind{lower semicontinuous} if
for every $y\in X$,
\[
\liminf_{x\to y}\varphi(x)\>\ge\>\varphi(y)\>.
\]
If $\varphi$ is lower semicontinuous and bounded from below, we call it a \textbf{Caristi--Kirk function on $X$}.
For a fixed Caristi--Kirk function $\varphi$ we consider \textbf{Caristi--Kirk balls} of the form
\begin{equation}                                \label{CKb}
B^\varphi_x\>:=\>\{y\in X\mid d(x,y)\le\varphi(x)-\varphi(y)\},\quad x\in X,
\end{equation}
and the corresponding \textbf{Caristi--Kirk ball space} $(X,\cB^\varphi)$ given by
\[
\cB^\varphi:=\{B^\varphi_x\mid x\in X\}.
\]
These ball spaces and their underlying theory can be employed to prove the Caristi--Kirk Theorem in a simple manner
(see below). We found the sets that we call Caristi--Kirk balls in a proof of the Caristi--Kirk Theorem given by
J.-P.~Penot in \cite{P}.

\parm
We say that a function $\phi:X\times X\to(-\infty,+\infty]$ is an \textbf{Oettli--Th\'era function on $X$} if it
satisfies the following conditions:
\begin{align*}
(a)\quad&\phi(x,\cdot):X\rightarrow(-\infty,+\infty]\text{ is lower semicontinous for all } x\in X;\\
(b)\quad&\phi(x,x)=0\text{ for all }x\in X;\\
(c)\quad&\phi(x,y)\le \phi(x,z)+\phi(z,y)\text{ for all }x,y,z\in X;\\
(d)\quad&\text{there exists }x_0\in X\text{ such that } \inf_{x\in X}\phi(x_0,x)>-\infty.
\end{align*}
This notion was, to our knowledge, first introduced by Oettli and Th\'era in \cite{OT}. An Oettli--Th\'era function
$\phi$ yields balls of the form
\[
B^\phi_x:=\{y\in X\mid d(x,y)\le-\phi(x,y)\},\quad x\in X,
\]
which will be called \textbf{Oettli--Th\'era balls}. If an element $x_0$ satisfies condition $(d)$ above, then
we will call it an
\bfind{Oettli--Th\'era element for $\phi$ in $X$}. For a fixed Oettli--Th\'era element $x_0$ we define the
associated \textbf{Oettli--Th\'era ball space} to be $(B^\phi_{x_0},\cB^\phi_{x_0})$, where
\[
\cB^\phi_{x_0}:=\{B^\phi_x\mid x\in B^\phi_{x_0}\}.
\]
We observe that for a given Caristi--Kirk function $\varphi:X\to\R$, the mapping
\[
\phi(x,y):=\varphi(y)-\varphi(x)
\]
is an Oettli--Th\'era function. Furthermore, every Caristi--Kirk ball is also an Oettli--Th\'era ball.

In general the balls defined above are not metric balls. However, when working in complete metric spaces they prove
to be a more useful tool than metric balls. As observed in the previous section, the completeness of a metric space
need not imply spherical completeness of the space of metric balls $(X,\cB_d)$. In the case of Caristi--Kirk and
Oettli--Th\'era balls, completeness turns
out to be equivalent to spherical completeness, as shown in the following two propositions.

\begin{proposition}                                            \label{pr-ck}
Let $(X,d)$ be a metric space. Then the following assertions are equivalent:
\sn
a) The metric space $(X,d)$ is complete.
\sn
b) Every Caristi--Kirk ball space $(X,\cB^\varphi)$ is spherically complete.
\sn
c) For every continuous function $\varphi:X\to\R$ bounded from below, the Caristi--Kirk ball space
$(X,\cB^\varphi)$ is spherically complete.
\end{proposition}

\begin{proposition}                                             \label{pr-ot}
A metric space $(X,d)$ is complete if and only if the Oettli--Th\'era ball space $(B^\phi_{x_0},\cB^\phi_{x_0})$
is spherically complete for every Oettli--Th\'era function $\phi$ on $X$ and every Oettli--Th\'era element $x_0$
for $\phi$ in $X$.
\end{proposition}

The proofs of Proposition \ref{pr-ck} and Proposition \ref{pr-ot} can be found in \cite[Proposition 3]{KKP} and
in \cite{BCLS}, respectively.

\parm
The easy proof of the next proposition is provided in \cite{BCLS}.
\begin{proposition}                                  \label{CKOTsc}
Every Caristi--Kirk ball space $(X,\cB^\varphi)$ and every Oettli--Th\'era ball space
$(B^\phi_{x_0},\cB^\phi_{x_0})$ is a strongly contractive normalized B$_x$--ball space.
\end{proposition}
\n
We will meet another strongly contractive ball space in the case of partially ordered sets; see
Proposition~\ref{pfssc}.

\pars
The following is the \bfind{Caristi--Kirk Fixed Point Theorem}:
\begin{theorem}                                       \label{CKFPT}
Take a complete metric space $(X,d)$ and a lower semicontinuous function $\varphi:X\to\R$ which is
bounded from below. If a function $f:X\to X$ satisfies the \bfind{Caristi condition}
\begin{equation}                           \label{CC}
d(x,fx)\>\le\>\varphi(x)-\varphi(fx)\>,
\end{equation}
for all $x\in X$, then $f$ has a fixed point on $X$.
\end{theorem}

Also in \cite{BCLS}, the same tools (with Proposition~\ref{pr-ck} replaced by Proposition~\ref{pr-ot}) are
used to prove the following generalization:
\begin{theorem}                                   \label{OTFPT}
Take a complete metric space $(X,d)$ and $\phi$ an Oettli-Th\'era function on $X$. If a function $f:X\to X$
satisfies
\begin{equation}                           \label{OTC}
d(x,fx)\>\le\>-\phi(x,fx),
\end{equation}
for all $x\in X$, then $f$ has a fixed point on $X$.
\end{theorem}
The conditions (\ref{CC}) and (\ref{OTC}) guarantee that $fx\in B_x$ for every $B_x\in\cB^\varphi$ or
$B_x\in\cB^\phi_{x_0}$, respectively. Hence Theorem~\ref{GFPT3} in conjunction with Propositions~\ref{pr-ck},
\ref{pr-ot} and~\ref{CKOTsc} proves Theorems~\ref{CKFPT} and~\ref{OTFPT}. Similar proofs were given in
\cite{BCLS} (see also \cite{KKP}). Note that conditions (\ref{CC}) and (\ref{OTC}) do not necessarily imply that
every ball $B_x$ is $f$-closed.

A variant of part 2) of Theorem~\ref{MTscsc} is used in \cite{BCLS} to give quick proofs of several theorems
that are known to be equivalent to the Caristi--Kirk Fixed Point Theorem (see \cite{OT,P,P2} for presentations of
these equivalent results and generalizations).

\pars
\begin{remark}                             \label{remax}
Assume that $(X,\cB)$ is a contractive B$_x$--ball space. Then we can define a partial
ordering on $X$ by setting
\[
x\>\prec\> y\>:\Leftrightarrow\> B_y\subsetneq B_x\>.
\]
If $(X,\cB)$ is strongly contractive, then the function $x\mapsto B_x$ is injective, and $X$ together with the
reverse of the partial order we have defined is order isomorphic to $\cB$ with inclusion, that is, the function
$x\mapsto B_x$ is an order isomorphism from $(X,\prec)$ onto $(\cB,<)$ where the latter is defined as in
the beginning of Section~\ref{sectZL}.

If the $B_x$ are the Caristi--Kirk balls defined in (\ref{CKb}), then we have that
\[
x\>\prec\> y\>\Leftrightarrow\> d(x,y)<\varphi(x)-\varphi(y)\>,
\]
which means that $\prec$ is the \bfind{Br{\o}nsted ordering} on $X$. The \bfind{Ekeland Variational Principle}
(cf.\ \cite{BCLS}) states that if the metric space is complete, then $(X,\prec)$ admits maximal elements, or in
other words, $\cB$ admits minimal balls. The Br{\o}nsted ordering has been used in several different proofs of the
Caristi--Kirk Fixed Point Theorem. However, at least in the proofs that also define and use the Caristi--Kirk balls
(such as the one of Penot in \cite{P}), it makes more sense to use directly their natural partial
ordering (as done in \cite{KKP}). But the main incentive to use the balls instead of the ordering is that it
naturally subsumes the metric
case in the framework of fixed point theorems in several other areas of mathematics which is provided by the
general theory of ball spaces as laid out in the present paper (see also \cite{[KK1],[KK2],[KKSo]}).

It has been shown that the Ekeland Variational Principle can be proven in the Zermelo Fraenkel axiom system ZF
plus the \bfind{axiom of dependent choice} DC which covers the usual mathematical induction (but not transfinite
induction, which is equivalent to the full axiom of choice). Conversely, it has been shown in \cite{Br} that the
Ekeland Variational Principle implies the axiom of dependent choice.

Several proofs have been provided for the Caristi--Kirk FPT that work in ZF+DC. Kozlowski has
given a proof that is \bfind{purely metric} as defined in his paper \cite{Ko}, which implies that the proof works
in ZF+DC. The proofs of Proposition~\ref{pr-ck} in \cite{KKP} and of Proposition~\ref{pr-ot} in \cite{BCLS} are
purely metric. The existence of singleton balls in Caristi--Kirk and Oettli-Th\'era ball spaces over complete
metric spaces can also be shown directly by purely metric proofs and this result can be used to give quick proofs
of many principles that are equivalent to the Caristi--Kirk FPT in ZF+DC (cf.\ \cite{BCLS}).
However, in other settings it may not be possible to deduce the existence in ZF+DC, so then the axiom of choice is
needed. Therefore, in view of the number of possible applications even beyond the scope as presented in this paper,
we do not hesitate to use Zorn's Lemma for the proofs of our generic fixed point theorems.

We should point out that proofs have been given that apparently prove the Caristi--Kirk FPT in ZF (see
\cite{Man,J2}). This means that the Caristi--Kirk FPT and the Ekeland Variational Principle are
equivalent in ZF+DC, but {\it not} in ZF. For the topic of axiomatic strength, see the discussions in
\cite{J,Ki,Ko}.
\end{remark}

\mn
%
%
\subsection{Totally ordered sets, abelian groups and fields}         \label{sectCK}
\mbox{ }\sn
Take any ordered set $(I,<)$. We define the \bfind{closed interval ball space} associated with $(I,<)$ to be
$(I,\cBcb)$ where $\cBcb$ consists of all closed intervals $[a,b]$ with $a,b\in I$. By a \bfind{cut} in
$(I,<)$ we mean a partition $(C,D)$ of $I$ such that $c<d$ for all $c\in C$, $d\in D$ and $C,D$ are nonempty.
The \bfind{cofinality} of a totally ordered set is the least cardinality of all cofinal subsets,
and the \bfind{coinitiality} of a totally ordered set is the cofinality of this set under the reverse ordering.
A cut $(C,D)$ is \bfind{asymmetric} if the cofinality of $C$ is different from the coinitiality of $D$. For
example, every cut in $\R$ is asymmetric. The following fact was first proved in \cite{[S]} for ordered fields,
and then in \cite{[KKSh]} for any totally ordered sets.
\begin{lemma}
The ball space $(I,\cBcb)$ associated with the totally ordered set
$(I,<)$ is spherically complete if and only if  every cut $(C,D)$ in $(I,<)$ is asymmetric.
\end{lemma}

Totally ordered sets, abelian groups or fields whose cuts are all asymmetric are called
\bfind{symmetrically complete}. By our above remark, $\R$ is symmetrically complete. The following theorem
was proved in \cite{[KKSh]}; its first assertion follows from the previous lemma. The second assertion addresses
the \bfind{natural valuation} of an ordered abelian group or field, which is the finest valuation compatible with
the ordering; it is nontrivial if and only if the ordering is nonarchimedean.
\begin{theorem}                         \label{symcompsphcomp}
A totally ordered set, abelian group or field is symmetrically complete if and only if its associated closed
interval ball space is spherically complete. The ultrametric ball space associated with the natural valuation
of a symmetrically complete ordered abelian group or field is a spherically complete ball space.
\end{theorem}

In \cite{[S]} it was shown that arbitrarily large symmetrically complete ordered
fields exist. With a different construction idea, this was reproved and
generalized in \cite{[KKSh]} to the case of ordered abelian groups and
totally ordered sets, and a characterization of symmetrically complete
ordered abelian groups and fields has been given.

In order to give an example of a fixed point theorem that can be proven in this setting,
it is enough to consider symmetrically complete ordered abelian groups, as the additive group of a
symmetrically complete ordered field is a symmetrically complete ordered abelian group. The following is
Theorem~21 of \cite{[KK1]} (see also \cite{[KKSh]}).
\begin{theorem}
Take an ordered abelian group $(G,<)$ and a function $f:G\rightarrow G$.
Assume that every nonempty chain of closed intervals in $G$ has
nonempty intersection and that $f$ has the following properties:
\sn
1) \ $f$ is nonexpanding:
\[
|fx-fy|\>\leq\> |x-y| \> \mbox{ for all } x,y\in G\>,
\]
2) \ $f$ is contracting on orbits: there is a positive rational number
$\frac{m}{n}<1$ with $m,n\in\N$ such that
\[
n|fx-f^2x| \>\leq\> m|x-fx| \> \mbox{ for all } x\in G\>.
\]
Then $f$ has a fixed point.
\end{theorem}

\parm
As in the case of ultrametric spaces, all singletons in $\cBcb$ are balls: $\{a\}=
[a,a]$. So also here, $(I,\cBcb)$ is {\bf S}$_2$ as soon as it is
{\bf S}$_1\,$. But again as in the case of ultrametric spaces,
{\bf S}$_2$ does not necessarily imply {\bf S}$_5$ or even
{\bf S}$_3\,$. For example, consider a nonarchimedean ordered
symmetrically complete field. The set of infinitesimals is the
intersection of balls $[-a,a]$ where $a$ runs through all positive
elements that are not infinitesimals. This intersection is not a ball,
nor is there a largest ball contained in it.

Further, we note:
\begin{lemma}                        \label{cc}
Assume that $(I,<)$ is a totally ordered set and its associated ball space
$(I,\cBcb)$ is an {\bf S}$_1^d$ or {\bf S}$_3$ ball space. Then $(I,<)$ is cut complete, that is,
for every cut $(C,D)$ in $(I,<)$, $C$ has a largest or $D$ has a smallest element.
\end{lemma}
\begin{proof}
First assume that $(I,\cBcb)$ is an {\bf S}$_1^d$ ball space, and take a cut $(C,D)$ in $I$. If $a,c\in C$ and
$b,d\in D$, then $\max\{a,c\}\in C$ and $\min\{b,d\}\in D$ and $[a,b]\cap [c,d]=[\max\{a,c\}, \min\{b,d\}]$. This
shows that
\[
\{[c,d]\mid c\in C\,,\, b\in D\}
\]
is a directed system in $\cBcb\,$. Hence its intersection is nonempty; if $a$ is contained in this intersetion,
it must be the largest element of $C$ or the least element of $D$. Hence $(I,<)$ is cut complete.

\pars
Now assume that $(I,<)$ is not cut complete; we wish to show that $(I,\cBcb)$ is not an {\bf S}$_3$ ball space.
Take a cut $(C,D)$ in $I$ such that $C$ has no largest element and $D$ has no least element. Pick some $c\in C$. Then
\[
\{[c,d]\mid d\in D\}
\]
is a nest of balls in $(I,\cBcb)$. Its intersection is the set $\{a\in C\mid c\leq a\}$. Since $C$ has no largest
element, this set does not contain a maximal ball. This shows that $(I,\cBcb)$ is not an {\bf S}$_3$ ball space.
\end{proof}

\pars
It is a well known fact that the only cut complete densely ordered
abelian group or ordered field is $\R$. So we have:

\begin{proposition}                        \label{limitedS}
The associated ball space of the reals is {\bf S}$^*\,$. For all other
densely ordered abelian groups and ordered fields the associated ball space can
at best be {\bf S}$_2\,$.
\end{proposition}
\begin{proof}
Take any centered system $\{[a_i,b_i]\mid i\in I\}$ of intervals in
$\R$. We set $a:=\sup_{i\in I} a_i$ and $b:=\inf_{i\in I} b_i\,$. Then
\[
\bigcap_{i\in I} [a_i,b_i]\>=\> [a,b]\>.
\]
We have to show that $[a,b]\ne\emptyset$, i.e., $a\leq b$. Suppose that
$a>b$. Then there are $i,j\in I$ such that $a_i>b_j\,$. But by
assumption, $[a_i,b_i]\cap [a_j,b_j]\ne\emptyset$, a contradiction. We
have now proved that the associated ball space of the reals is {\bf S}$^*\,$.

\pars
The second assertion follows from Lemma~\ref{cc}.
\end{proof}

\mn
%
%
\subsection{Topological spaces}
\mbox{ }\sn
If $\cal X$ is a topological space on a set $X$, then we will take its associated ball space to be $(X,\cB)$ where
$\cB$ consists of all nonempty closed sets. Since the intersections of arbitrary collections of closed sets are
again closed, this ball space is intersection closed.
\pars
The following theorem shows how compact topological spaces are characterized by the properties of their
associated ball spaces; note that we use ``compact'' in the sense of ``quasi-compact'', that is, it does not
imply the topology being Hausdorff.
\begin{theorem}                             \label{cts}
The following are equivalent for a topological space ${\cal X}$:
\sn
a) \ ${\cal X}$ is compact,
\sn
b) \ the nonempty closed sets in ${\cal X}$ form an {\bf S}$_1$ ball
space,
\sn
c) \ the nonempty closed sets in ${\cal X}$ form an {\bf S}$^*$ ball space.
\end{theorem}
\begin{proof}
a) $\Rightarrow$ b): \ Assume that ${\cal X}$ is compact. Take a nest $(X_i)_{i\in I}$ of balls in
$(X,\cB)$ and suppose that $\bigcap_{i\in I} X_i =\emptyset$. Then $\bigcup_{i\in I} X\setminus X_i = X$, so
$\{X\setminus X_i\mid i\in I\}$ is an open cover of $\cal X$. It follows that there are $i_1,\ldots, i_n\in I$
such that $X\setminus X_{i_1}\cup\ldots\cup X\setminus X_{i_n}=X$, whence $X_{i_1}\cap\ldots\cap X_{i_n}=
\emptyset$. But since the $X_i$ form a nest, this intersection equals the smallest of the $X_{i_j}$, which is
nonempty. This contradiction proves that the nonempty closed sets in ${\cal X}$ form an {\bf S}$_1$ ball
space.

\mn
b) $\Rightarrow$ c): \ This follows from Theorem~\ref{sccent}.

\mn
c) $\Rightarrow$ a): \ Assume that the nonempty closed sets in ${\cal X}$ form an {\bf S}$^*$ ball
space. Take an open cover $Y_i$, $i\in I$, of $\cal X$. Since $\bigcup_{i\in I} Y_i =X$, we have that
$\bigcap_{i\in I} X\setminus Y_i = \emptyset$. As the ball space is {\bf S}$^*$, this means that
$\{X\setminus Y_i\mid i\in I\}$ cannot be a centered system. Consequently, there are $i_1,\ldots, i_n\in I$
such that $X\setminus Y_{i_1}\cap\ldots\cap X\setminus Y_{i_n}=\emptyset$, whence $Y_{i_1}\cup\ldots\cup Y_{i_n}=X$.
\end{proof}

\parm
Some of the assertions of the following topological fixed point theorems were already proven in
\cite[Theorem~11]{[KK1]}. We will give their modified and improved proofs here as they illustrate applications
of Theorems~\ref{GFPT2} and~\ref{basic1b}.
\begin{theorem}                             \label{top3}
Take a compact space $X$ and a closed function $f:X\rightarrow X$.
Assume that for every $x\in X$ with $fx\ne x$ there is a closed subset
$B$ of $X$ such that $x\in B$ and $x\notin f(B)\subseteq B$. Then $f$
has a fixed point in $B$.
%
\end{theorem}
\begin{proof}
For every $x\in X$ we consider the following family of
balls:
\[
\mathfrak B_x := \{B \mid B \text{ closed subset of $X$, }
x\in B \text{ and } f(B)\subseteq B\}.
\]
Note that $\mathfrak B_x$ is nonempty because it contains $X$.
We define
\begin{equation}                       \label{topB_x}
B_x := \bigcap\mathfrak B_x\>.
\end{equation}
We see that $x \in B_x$ and that $f(B_x)\subseteq B_x$ by part 2) of Lemma~\ref{observe}. Further, $B_x$
is closed, being the intersection of closed sets. This shows that $B_x$
is the smallest member of $\mathfrak B_x\,$.

For every $B\in \mathfrak B_x$ we have that $fx\in B$ and therefore,
$B\in \mathfrak B_{fx}\,$. Hence we find that $B_{fx} \subseteq B_x$.

Assume that $fx\ne x$. Then by hypothesis, there is a closed set $B$ in
$X$ such that $x\in B$ and $x\notin f(B)\subseteq B$. Since $f$ is a
closed function, $f(B)$ is closed. Moreover, $f(f(B)) \subseteq f(B)$
and $fx \in f(B)$, so $f(B) \in \mathfrak B_{fx}$. Since $x\notin f(B)$,
we conclude that $x\notin B_{fx}$, whence $B_{fx} \subsetneq B_x\,$. We
have proved that $f$ is contracting on orbits.
Our theorem now follows from Theorem~\ref{GFPT2} in conjunction with Theorem~\ref{cts}.
\end{proof}
\sn
Note that if $B$ satisfies the assumptions of the theorem, then
$B\in \mathfrak B_x\,$. Hence the set $B_x$ defined in (\ref{topB_x}) satisfies
$B_x\subseteq B$, $f(B_x)\subseteq f(B)$ and therefore $x\notin f(B_x)$. This shows that $B_x$ is the smallest of
all closed subsets $B$ of $X$ for which $x\in B$ and $x\notin f(B)\subseteq B$.

An interesting interpretation of
the ball $B_x$ defined in (\ref{topB_x}) will be given in Remark~\ref{topB_xrem} below.

\parm
The next theorem follows immediately from part 1) of Theorem~\ref{basic1b} in conjunction with Theorem~\ref{cts}.
\begin{theorem}                             \label{topn}
Take a compact space $X$ and a closed function $f:X\rightarrow X$.\sn
1) \ If every nonempty closed and $f$-closed subset $B$ of $X$ contains
a closed $f$-contracting subset, then $f$ has a fixed point in $X$. \sn
2) \ If every nonempty closed and $f$-closed subset $B$ of $X$ is
$f$-contracting, then $f$ has a unique fixed point in $X$.
\end{theorem}

\pars
The condition that every $f$-closed ball is $f$-contracting may appear to be quite strong. Yet there is a natural
example in the setting of topological spaces where this condition is satisfied in a suitable collection of closed
sets. In \cite{[SWJ]}, Steprans, Watson and Just define the notion of ``$J$-contraction'' for a
continuous function $f:X\to X$ on a topological space $X$ as follows.
An open cover $\mathcal U$ of $X$ is called \bfind{$J$-contractive for $f$} if for every $U\in \mathcal U$ there
is $U' \in \mathcal U$ such that the image of the closure of $U$ under $f$ is a subset of $U'$. Then $f$ is called
a \bfind{$J$-contraction} if any open cover $\mathcal U$  has a {$J$-contractive} refinement for $f$.
We will use two important facts about $J$-contractions $f$ on a connected compact Hausdorff space $X$ which the
authors prove in the cited paper:
\begin{enumerate}
\item[(J1)] If $B$ is a closed subset of $X$ with $f(B)\subseteq B$,
then the restriction of $f$ to $B$ is also a $J$-contraction (\cite[Proposition 1, p.\ 552]{[SWJ]});
\item[(J2)] If $f$ is onto, then $|X|$ = 1 (\cite[Proposition 4, p.\ 554]{[SWJ]}).
\end{enumerate}

The following is Theorem 4 of \cite{[SWJ]}:

\begin{theorem}                             \label{SWJ}
Take a connected compact Hausdorff space $X$ and a continuous
$J$-contraction $f:X\rightarrow X$. Then $f$ has a unique fixed point.
\end{theorem}

We will deduce our theorem from Theorem~\ref{basic1b}. We take $\cB$ to be the set of all nonempty closed
connected subsets of $X$; in particular, $X\in\cB$. Take any $B\in\cB$. As $f$ is a continuous function on the
compact Hausdorff space $X$, it is a closed function, so $f(B)$ is closed. Since $B$ is connected and
$f$ is continuous, $f(B)$ is also connected. Hence $f(B)\in\cB$.

Further, the intersection of any chain of closed connected subsets of $X$ is closed and connected. This shows
that $\cB$ is chain intersection closed. By Theorem~\ref{cts} the ball space consisting of all nonempty closed
subsets of the compact space $X$ is {\bf S}$^*$. As it contains $\cB$, $(X,\cB)$ is {\bf S}$_1$ and it follows
from Proposition~\ref{cicSeq} that $(X,\cB)$ is an {\bf S}$_5$ ball space.

Finally, we have to show that every $f$-closed ball $B\in\cB$ is $f$-contracting. As $B$ is closed in $X$, it is
also compact Hausdorff, and it is connected as it is a ball in $\cB$. By (J1), the restriction of $f$ to $B$ is
also a $J$-contraction. Therefore, we can replace $X$ by $B$ and apply (J2) to find that if $f$ is onto, then $B$
is a singleton; this shows that $B$ is $f$-contracting. Now Theorem~\ref{SWJ} follows from part 2) of
Theorem~\ref{basic1b} as desired.

\parm
It should be noted that $J$-contractions appear in a natural way in the metric setting. The following is the
content of Theorems~2 and~3 of \cite{[SWJ]}:
\begin{theorem}
Any contraction on a compact metric space is a $J$-contrac\-tion. Conversely, if $f$ is a $J$-contraction on a
connected compact metrizable space $X$, then $X$ admits a metric under which $f$ is a contraction.
\end{theorem}

\mn
%
%
\subsection{Partially ordered sets}       \label{sectpos}
\mbox{ }\sn
Take any nonempty partially ordered set $(T,<)$. We will associate with it two different ball
spaces; first, the ball space of principal final segments, and then later the segment ball space.

A \bfind{principal final segment} is a set $[a,\infty):=\{c\in T\mid a\leq c\}$ with $a\in T$.
Then the \bfind{ball space of principal final segments} is $(T,\cBpfs)$ where $\cBpfs:=
\{[a,\infty)\mid a \in T\}$. The following proposition gives the interpretation of spherical
completeness for this ball space; we leave its straightforward proof to the reader.
\begin{proposition}                             \label{io}
The following assertions are equivalent:
\sn
a) \ the poset $(T,<)$ is inductively ordered,
\sn
b) \ the ball space $(T,\cBpfs)$ is an {\bf S}$_1$ ball space,
\sn
c) \ $(T,\cBpfs)$ is an {\bf S}$_2$ ball space.
\end{proposition}

\pars
We also leave it to the reader to show that $(T,\cBpfs)$ is an {\bf S}$_3$ (or {\bf S}$_3^d$ or
{\bf S}$_3^c$) ball space if and only if every chain (or directed system, or centered system,
respectively) has minimal upper bounds.

\parm
If $\{a_i\mid i\in I\}$ is a subset of $T$, then $\sup_{i\in I} a_i$ will denote its supremum,
if it exists. We will need the following fact, whose proof we again leave to the reader.
\begin{lemma}                               \label{prelpo}
The equality
\[
[a,\infty)=\bigcap_{i\in I} \, [a_i,\infty)
\]
holds if and only if $a=\sup_{i\in I} a_i\,$. Further, $\bigcap_{i\in I}^{}
[a_i,\infty)$ is the (possibly empty) set of all upper bounds for $\{a_i\mid i\in I\}$.
\end{lemma}

\pars
An element $a$ in a poset is called \bfind{top element} if $b\leq a$ for all elements $b$ in the poset,
and \bfind{bottom element} if $b\geq a$ for all elements $b$ in the poset. A top element is commonly denoted by
\gloss{$\top$}, and a bottom element by \gloss{$\bot$}. A poset $(T,<)$ is an \bfind{upper semilattice} if
every two elements in $T$ have a supremum, and a \bfind{complete upper semilattice} if
every nonempty set of elements in $T$ has a supremum.
\begin{proposition}                         \label{propo}
1) \ $(T,\cBpfs)$ is finitely intersection closed if and only if every
nonempty finite bounded subset of $T$ has a supremum.
\sn
2) \ $(T,\cBpfs)$ is intersection closed if and only if every nonempty bounded
subset of $T$ has a supremum, i.e., $(T,<)$ is bounded complete.
\sn
3) \ If $(T,<)$ has a top element, then $(T,<)$ is an upper semilattice if and only if $(T,\cBpfs)$ is finitely
intersection closed,
\sn
4) \ $(T,<)$ is a complete upper semilattice if and only if $(T,<)$ has a top element and $(T,\cBpfs)$ is
intersection closed.
\end{proposition}
\begin{proof}
1), 2): \ Assume that $(T,\cBpfs)$ is (finitely) intersection closed and take a nonempty (finite) subset
$\{a_i\mid i\in I\}$ of $T$. If this set is bounded, then $\bigcap_{i\in I}^{} [a_i,\infty)$ is nonempty, and thus
by assumption it is equal to $[a,\infty)$ for some $a\in T$. By Lemma~\ref{prelpo}, this implies that $a=\sup_{i
\in I} a_i\,$, showing that $\{a_i\mid i\in I\}$ has a supremum.

Now assume that every nonempty (finite) bounded subset of $T$ has a supremum. Take a nonempty (finite) set
$\{[a_i,\infty)\mid i\in I\}$ of balls in $\cBpfs$ with nonempty intersection. Take $b\in\bigcap_{i\in I}
[a_i,\infty)$.
Then $b$ is an upper bound of $\{a_i\mid i\in I\}$. By assumption, there exists $a=\sup_{i\in I} a_i$ in $T$.
Again by Lemma~\ref{prelpo}, this implies that $\bigcap_{i\in I}^{} [a_i,\infty)=[a,\infty)$. Hence,
$(T,\cBpfs)$ is (finitely) intersection closed.
\mn
3) and 4) follow from 1) and 2), respectively, because if $(T,<)$ has a
top element, then every nonempty subset is bounded.
\end{proof}

\parm
We add to our hierarchy~(\ref{hier}) an even stronger property: we say that the ball space $(X,\cB)$ is an
{\bf S}$^{**}$ ball space if $\cB$ is closed under {\emph arbitrary} intersections; in particular, this implies
that intersections of arbitrary collections of balls are nonempty. Every {\bf S}$^{**}$ ball
space is an {\bf S}$^*$ ball space. Note that every complete upper semilattice has a top element.
\begin{proposition}                         \label{propo2}
1) \ Assume that $(T,<)$ has a top element $\top$. Then every intersection of balls in $(T,\cBpfs)$ contains the
ball $[\top,\infty)$, and $(T,\cBpfs)$ is an {\bf S}$_2^c$ ball space. Moreover, $(T,\cBpfs)$ is an {\bf S}$^*$
ball space if and only if it is an {\bf S}$^{**}$ ball space.
\sn
2) \ $(T,\cBpfs)$ is an {\bf S}$^{**}$ ball space if and only if $(T,<)$ has a top element and $(T,\cBpfs)$ is
intersection closed.
\sn
3) $(T,<)$ is a complete upper semilattice if and only if $(T,\cBpfs)$ is an {\bf S}$^{**}$ ball space.
\end{proposition}
\begin{proof}
1): The first two statements are obvious. If $(T,<)$ has a top element, then every collection of balls in
$\cBpfs$ is a centered system. Hence if $(T,\cBpfs)$ is an {\bf S}$^*$ ball space, then it is an {\bf S}$^{**}$
ball space.
\sn
2): Assume that $(T,\cBpfs)$ is an {\bf S}$^{**}$ ball space. Then it follows directly from the definition that it
is intersection closed. Further, the intersection over $\{[a,\infty)\mid
a\in T\}$ is a ball $[b,\infty)$. By Lemma~\ref{prelpo}, $b$ is the supremum of $T$ and thus a top element.

Now assume that $(T,<)$ has a top element $\top$ and $(T,\cBpfs)$ is intersection closed, and take an arbitrary
collection of balls in $\cBpfs$. As all of the balls contain $\top$, their intersection is nonempty, and hence
by our assumption, it is a ball.
\sn
3): This follows from part 2) of our proposition together with part 4) of Proposition~\ref{propo}.
\end{proof}

\pars
Now we can characterize chain complete and directed complete posets by properties from our hierarchy:
\begin{theorem}
Take a poset $(T,<)$. Then the following are equivalent:
\sn
a) \ $(T,<)$ is chain complete,
\sn
b) \ $(T,<)$ is directed complete,
\sn
c) \ $(T,\cBpfs)$ is an {\bf S}$_5$ ball space,
\sn
d) \ $(T,\cBpfs)$ is an {\bf S}$_5^d$ ball space.

\pars
If every nonempty finite bounded subset of $T$ has a supremum, then the above
properties are also equivalent to
\sn
e) \ $(T,\cBpfs)$ is an {\bf S}$^*$ ball space.
\end{theorem}
\begin{proof}
The equivalence of assertions a) and b) follows from Proposition~\ref{cc=dc}.
\sn
b) $\Rightarrow$ d): \ Assume that $(T,<)$ is directed complete and take
a directed system $S=\{[a_i,\infty)\mid i\in I\}$ in $\cBpfs\,$. Then also $\{a_i\mid i\in I\}$ is a
directed system in $(T,<)$. By our assumption on $(T,<)$ it follows that
$\{a_i\mid i\in I\}$ has a supremum $a$ in $T$. By
Lemma~\ref{prelpo}, $[a,\infty)=\bigcap_{i\in I}^{} [a_i,\infty)$, which
shows that the intersection of $S$ is a ball.
\sn
d) $\Rightarrow$ c) holds by the general properties of the hierarchy.
\sn
c) $\Rightarrow$ a): \ Take a chain $\{a_i\mid i\in I\}$ in $T$. Since $(T,\cBpfs)$ is  an {\bf S}$_5$ ball space,
the intersection of the nest $\cN=([a_i,\infty))_{i\in I}$ is a ball $[a,\infty)$. It follows by Lemma~\ref{prelpo}
that $a$ is the supremum of the chain, which proves that $(T,<)$ is chain complete.

\parm
If every nonempty finite bounded subset of $T$ has a supremum, then by part 1) of
Proposition~\ref{propo}, $(T,\cBpfs)$ is finitely intersection closed,
hence by Proposition~\ref{fichier}, properties {\bf S}$_5^d$ and {\bf S}$^*$ are equivalent.
\end{proof}

\begin{remark}
Note that we define chains to be \emph{nonempty} totally ordered sets and similarly, consider directed systems to
be  nonempty. If we drop this convention, then the theorem remains true if we require in c) and d) that $(T,<)$ has
a least element.
\end{remark}

\parm
The ball space $(T,\cBpfs)$ shares an important property with Caristi--Kirk and Oettli--Th\'era
ball spaces, as shown by the next proposition, whose straightforward proof we omit.
\begin{proposition}                          \label{pfssc}
The ball space $(T,\cBpfs)$ is a normalized strongly contractive B$_x$-ball space, where
\[
B_x\>:=\> [x,\infty)\>\in\,\cBpfs\>.
\]
\end{proposition}

\parm
A function $f$ on a poset $(T,<)$ is \bfind{increasing} if $f(x)\geq x$ for all $x\in T$. The following
result is an immediate consequence of Zorn's Lemma, but can also be seen as a
corollary to Propositions~\ref{io} and~\ref{pfssc} together with Theorem~\ref{GFPT3}:
\begin{theorem}                               \label{wBW}
Every increasing function $f:X\rightarrow X$ on an inductively ordered poset $(T,<)$ has a fixed
point.
\end{theorem}

Note that this theorem implies the \bfind{Bourbaki-Witt Theorem} (see \cite{Bou,Wi} or the short
description on Wikipedia), which differs from it by assuming that every chain in $(T,<)$ even has
a supremum.

\pars
A function $f$ on a poset $(T,<)$ is called \bfind{order preserving} if $x\leq y$ implies
$fx\leq fy$. The following result  is an easy consequence of Theorem~\ref{wBW}:
\begin{theorem}
Take an order preserving function $f$ on a nonempty poset $(T,<)$ which contains at least one
$x$ such that $fx\geq x$ (in particular, this holds when $(T,<)$ has a bottom element). Assume
that $(T,<)$ is chain complete. Then $f$ has a fixed point.
\end{theorem}
\begin{proof}
Take $S:=\{x\in T\mid fx\geq x\}\ne
\emptyset$. Then also $S$ is chain complete. Indeed, if $(x_i)_{i\in I}$ is a chain in $S$, hence
also in $T$, then it has a supremum $z\in T$ by assumption. Since $z\geq x_i$ and $f$ is order
preserving, we have that $fz\geq fx_i\geq x_i$ for all $i\in I$, so $fz$ is also an upper bound
for $(x_i)_{i\in I}$. Therefore,
$fz\geq z$ since $z$ is the supremum of the chain, showing that $z\in S$.

Further, $S$ is closed under $f$, because if $x\in S$, then $fx\geq x$, hence $f^2x\geq fx$ since $f$ is assumed
to be order preserving; this shows that $fx\in S$ Now the existence of a fixed point follows from
Theorem~\ref{wBW}.
\end{proof}

\parb
The second ball space we associate with posets will be particularly useful for the study of lattices.
We define the \bfind{principal segment ball space} $(T,\cBiv)$ of the poset $(T,<)$ by taking $\cBiv$ to contain
all \bfind{principal segments} (which may also be called ``closed
convex subsets''), that is, the closed intervals $[a,b]:= \{c\in T\mid a\leq c\leq b\}$ for $a,b\in
T$ with $a\leq b$, the principal initial and final segments $\{c\in T\mid c\leq a\}$ and $\{c\in
T\mid a\leq c\}$ for $a\in T$, and $T$ itself. Note that all of these sets are of the
form $[a,b]$ if and only if $T$ has a top element $\top$ and a bottom
element $\bot$. Even if $T$ does not have these elements, we will still
use the notation $[\bot,b]$ for $\{c\in T\mid c\leq b\}$ and $[a,\top]$
for $\{c\in T\mid a\leq c\}$. Hence,
\[
\cBiv=\{[a,b]\mid a\in T\cup\{\bot\},\,b\in T\cup\{\top\}\}\>.
\]
If $\bot,\top\in T$ (as is the case for complete lattices), this is a generalization to posets of the closed
interval ball space $\cBcb$ that we defined for linearly ordered sets. We will thus also talk again of ``closed
intervals'' $[a,b]$.

A greatest lower bound of a subset $S$ of $T$ will also be called its \bfind{infimum}.
If $\{a_i\mid i\in I\}$ is a subset of $T$, then $\inf_{i\in I} a_i$ will denote its infimum, if it exists.
\begin{lemma}                                \label{int=ab}
Take subsets $\{a_i\mid i\in I\}$ and $\{b_i\mid i\in I\}$ of T such that $a_i\leq b_j$ for all $i,j\in I$.
If $a=\sup_{i\in I} a_i$ and $b=\inf_{i\in I} b_i$ exist, then $a\leq b$ and
\[
[a,b]\>=\>\bigcap_{i\in I}^{} [a_i,b_i]\>.
\]
\end{lemma}
\begin{proof}
We can write
\[
\bigcap_{i\in I}^{} [a_i,b_i]\>=\>\bigcap_{i\in I}^{} ([a_i,\top]\cap
[\bot,b_i])\>=\> \bigcap_{i\in I}^{} [a_i,\top] \>\cap\> \bigcap_{i\in I}^{} [\bot,b_i]
\]
Applying Lemma~\ref{prelpo}, we obtain that $[a,\top]= \bigcap_{i\in I} [a_i,\top]$, and applying it to
$L$ with the reverse order, we obtain that $[\bot,b]=\bigcap_{i\in I}^{} [\bot,b_i]$. Hence the above
intersection is equal to $[a,b]$, which we will now show to be nonempty.

By the assumption of our lemma, every $b_j$ is an upper bound of the set $\{a_i\mid i\in I\}$. Since $a$ is the
least upper bound of this set, we find that $a\leq b_i$ for all $i\in I$. As $b$ is the greatest lower bound of
the set $\{b_i\mid i\in I\}$, it follows that $a\leq b$.
\end{proof}

\mn
%
%
\subsection{Lattices}                          \label{sectlat}
\mbox{ }\sn
A \bfind{lattice} is a poset in which every two elements have a supremum and an
infimum (greatest lower bound). It then follows that all finite sets in a lattice $(L,<)$ have a supremum and an
infimum. A \bfind{complete lattice} is a poset in which all nonempty sets have a supremum and an
infimum. Lemma~\ref{int=ab} implies the following analogue to Proposition~\ref{propo}:

\begin{proposition}
The ball space $(L,\cBiv)$ associated to a lattice $(L,<)$ is finitely intersection closed.
The ball space $(L,\cBiv)$ associated to a complete lattice $(L,<)$ is intersection closed.
\end{proposition}

\pars
For a lattice $(L,<)$, we denote by $(L,>)$ the lattice endowed with the reverse order.
We will now characterize complete lattices by properties from our hierarchy.
\begin{theorem}                             \label{poulbS4c}
For a poset $(L,<)$, the following assertions are equivalent.
\sn
a) \ $(L,<)$ is a complete lattice,
\sn
b) \ $(L,<)$ and $(L,>)$ are complete upper semilattices,
\sn
c) \ the principal final segments of $(L,<)$ and of $(L,>)$ form {\bf S}$^{**}$ ball spaces,
\sn
d) \ $(L,\cBiv)$ is an {\bf S}$^*$ ball space and $(L,<)$ admits a top and a bottom element,
\sn
e) \ $(L,\cBiv)$ is an {\bf S}$^*$ ball space and every finite set in $(L,<)$ has an upper and a lower bound.
\end{theorem}
\begin{proof}
The equivalence of a) and b) follows directly from the definitions. The equivalence of b) and c) follows from part
3) of Proposition~\ref{propo2}.
\sn
a) $\Rightarrow$ d): \
Assume that $(L,<)$ is a complete lattice. Then it admits a top element (supremum of all its elements) and a
bottom element (infimum of all its elements).
Take a centered system $\{[a_i,b_i]\mid i\in I\}$ in $(L,\cBiv)$. Then for all $i,j\in I$, $[a_i,b_i]
\cap [a_j,b_j]\ne\emptyset$, so $a_i\leq b_j\,$.  Since $(L,<)$ is a complete lattice, $a:=\sup_{i\in I} a_i$
and $b:=\inf_{i\in I} b_i$ exist. From
Lemma~\ref{int=ab} it follows that $\bigcap_{i\in I}^{} [a_i,b_i]=[a,b]\ne\emptyset$, which consequently is a ball
in $\cBiv$. We have proved that $(L,\cBiv)$ is an {\bf S}$^*$ ball space.
\sn
d) $\Rightarrow$ e): \ A top element is an upper bound and a bottom element a lower bound for every set of
elements.
\sn
e) $\Rightarrow$ a): \
Take a poset $(L,<)$ that satisfies the assumptions of e), and
any subset $S\subseteq L$. If $S_0$ is a finite subset of $S$, then
it has an upper bound $b$ by assumption. Hence the balls $[a,\top]$,
$a\in S_0\,$, have a nonempty intersection, as it contains $b$. This
shows that $\{[a,\top]\mid a\in S\}$ is a centered system of balls.
Since $(L,\cBiv)$ is an {\bf S}$^*$ ball space, its intersection is a ball $[c,d]$, where we must have $d=\top$.
By Lemma~\ref{prelpo}, $c$ is the supremum of $S$.

Working with the reverse order, one similarly shows that $S$ has an
infimum since $(L,\cBiv)$ is an {\bf S}$^*$ ball space. Hence, $(L,<)$ is a complete lattice.
\end{proof}

\parm
For our next theorem, we will need one further lemma:

\begin{lemma}                               \label{lcomlat}
For a lattice $(L,<)$, the following are equivalent:
\sn
a) \ $(L,<)$ is a complete lattice,
\sn
b) \ $(L,<)$ and $(L,>)$ are directed complete posets,
\sn
c) \ $(L,<)$ and $(L,>)$ are chain complete posets.
\end{lemma}
\begin{proof}
The implication a) $\Rightarrow$ b) is trivial as every nonempty set in a complete
lattice has a supremum and an infimum.
\sn
b) $\Rightarrow$ a): \ Take a nonempty subset $S$ of $L$. Let $S'$ be
the closure of $S$ under suprema and infima of arbitrary finite subsets
of $S$. Then $S'$ is a directed system in both $(L,<)$ and $(L,>)$.
Hence by b), $S'$ has an infimum $a$ and a supremum $b$. These are lower
and upper bounds, respectively, for $S$. Suppose there was an upper
bound $c<b$ for $S$. Then there would be a supremum $d$ of some finite
subset of $S$ such that $d>c$. But as $c$ is also an upper bound of this
finite subset, we must have that $d\leq c$. This contradiction shows
that $b$ is also the supremum of $S$. Similarly, one shows that $a$ is
also the infimum of $S$. This proves that $(L,<)$ is a complete lattice.
\sn
b) $\Leftrightarrow$ c) follows from Proposition~\ref{cc=dc}.
\end{proof}

\parm
Now we can prove:
\begin{theorem}                             \label{charcola}
For a lattice $(L,<)$, the following are equivalent:
\sn
a) \ $(L,<)$ is a complete lattice,
\sn
b) \ $(L,\cBiv)$ is an {\bf S}$_5$ ball space,
\sn
c) \ $(L,\cBiv)$ is an {\bf S}$^*$ ball space.
\end{theorem}
\begin{proof}
a) $\Rightarrow$ c): This follows from Theorem~\ref{poulbS4c}.
\sn
c) $\Rightarrow$ b) holds by the general properties of the hierarchy.
\sn
b) $\Rightarrow$ a):
By Lemma~\ref{lcomlat} it suffices to prove that $(L,<)$ and $(L,>)$ are
chain complete posets. Take a chain $\{a_i\mid i\in I\}$ in $(L,<)$.
Then $\{[a_i,\top]\mid i\in I\}$ is a nest of balls in $(L,\cBiv)$.
Since $(L,\cBiv)$ is an {\bf S}$_5$ ball space, the intersection of this
nest is a ball $[a,b]$ for some $a,b\in L$; it must be of the form $[a,\top]$
since the intersection contains $\top$. From Lemma~\ref{prelpo} we
infer that $a=\sup_{i\in I} a_i\,$. This shows that $(L,<)$ is a chain
complete poset. The proof for $(L,>)$ is similar.
\end{proof}

\parm
An example of a fixed point theorem that holds in complete lattices is the Knaster--Tarski
Theorem, which we have mentioned in the Introduction.

\bn
%
%
\section{Spherical closures in {\bf S}$^*$ ball spaces}
%
%
\subsection{Spherical closures and subspaces}               \label{sectsub}
\mbox{ }\n
The particular strength of {\bf S}$^*$ ball spaces enables us to introduce a closure operation
similar to the topological closure. We will also introduce a notion of sub-ball space and show
that a sub-ball space of an {\bf S}$^*$ ball space will again be an {\bf S}$^*$ ball space.

In order to distinguish between a ball space on a set $X$ and one on a subset $Y$, we will use
the notations $\cB_X$ and $\cB_Y$, respectivly. As before, if $f:X\rightarrow X$ is a
function, then $\cB_X^f$ will denote the collection of all $f$-closed balls in $\cB_X\,$.
The next lemma presents a simple but useful observation. It follows from the fact that the
intersection over any collection of $f$-closed sets is again $f$-closed, see part 2) of
Lemma~\ref{observe}.
\begin{lemma}                               \label{S4cfcl}
If $(X,\cB_X)$ is an {\bf S}$^*$ ball space, then so is $(X,\cB_X^f)$, provided that
$\cB_X^f\ne\emptyset$.
\end{lemma}

For every nonempty subset $S$ of some ball in $\cB_X\,$, we define
\[
\scl_{\cB_X}(S)\>:=\> \bigcap\{B\in\cB_X\mid S\subseteq B\}
\]
and call it the \bfind{spherical closure} of $S$ in $\cB_X\,$.
\begin{lemma}                     \label{sclsmb}
Take an {\bf S}$^*$ ball space $(X,\cB_X)$.
\sn
1) \ For every nonempty subset $S$ of some ball in $\cB_X\,$, $\scl_{\cB_X}(S)$ is the smallest
ball in $\cB_X$ containing $S$.
\sn
2) \ If $f:X\rightarrow X$ is a function, then for every nonempty subset $S$ of some
$f$-closed ball in $\cB_X\,$, $\scl_{\cB_X^f}(S)$ is the smallest $f$-closed ball containing $S$.
\end{lemma}
\begin{proof}
1) \ The collection of all balls containing $S$ is nonempty by our condition that $S$ is a subset
of a ball in $\cB_X\,$. The intersection of this collection contains $S\ne\emptyset$, so it is a
centered system, and since $(X,\cB_X)$ is {\bf S}$^*$, its intersection is a ball. As all balls
containing $S$ appear in the system, the intersection must be the smallest ball containing $S$.
\sn
2) \ This follows from part 1) together with Lemma~\ref{S4cfcl}.
\end{proof}

\sn
Note that if $X\in\cB_X\,$, then we can drop the condition that $S$ is the subset of some ball (or some $f$-closed
ball, respectively) in $\cB_X\,$.

\pars
\begin{remark}                           \label{topB_xrem}
The ball $B_x$ defined in (\ref{topB_x}) in the proof of Theorem~\ref{top3} is equal to $\scl_{\cB^f}(\{x\})$,
where $\cB_X^f$ is the set of all closed $f$-closed sets of the topological space under consideration.
\end{remark}

\parm
The proof of the following observation is straightforward:

\begin{lemma}                               \label{hullinc}
Take an {\bf S}$^*$ ball space $(X,\cB_X)$. If $S\subseteq T$ are nonempty subsets of a ball in
$\cB_X\,$, then $\scl_{\cB_X}(S)\subseteq\scl_{\cB_X}(T)$.
\end{lemma}

\parm
For any subset $Y$ of $X$, we define:
\begin{equation}
\cB_X\cap Y\>:=\>\{B\cap Y\mid B\in \cB_X\}\setminus\{\emptyset\}\>.
\end{equation}
If there is at least one ball $B\in\cB_X$ such that $Y\cap B\ne\emptyset$, then $\cB_X\cap Y
\ne\emptyset$ and $(Y,\cB_X\cap Y)$ is a ball space.

\begin{lemma}                               \label{SC5sub}
Take an {\bf S}$^*$ ball space $(X,\cB_X)$ and a subset $Y\subseteq X$ such that $\cB_X\cap Y
\ne\emptyset$.
\sn
1) \ For each $B\in \cB_X\cap Y$,
\[
\scl_{\cX}(B)\cap Y\>=\>B\>.
\]
\sn
2) \ The function
\begin{equation}
\cB_X\cap Y\ni B\mapsto \scl_{\cX}(B)
\end{equation}
preserves inclusion in the strong sense that
\[
B_1\subseteq B_2\;\Longleftrightarrow\; \scl_{\cX}(B_1)
\subseteq\scl_{\cX}(B_2)\>.
\]
3) \ If $(B_i)_{i\in\N}$ is a centered system of balls in $(Y,\cB_X\cap
Y)$, then ($\scl_{\cX}(B_i))_{i\in I}$ is a centered system of balls in
$(X,\cB_X)$ with
\begin{equation}
\bigcap_{i\in I} B_i\>=\> \left(\bigcap_{i\in I}
\scl_{\cX}(B_i)\right)\,\cap\,Y\>.
\end{equation}
\end{lemma}
\begin{proof}
1): \ It follows from the definition of $\scl_{\cX}(B)$ that $B\subseteq
\scl_{\cX}(B)$, so $B\subseteq\scl_{\cX}(B)\cap Y$. Since $B\in \cB_X\cap Y$,
we can write $B=B'\cap Y$ for some $B'\in\cB_X\,$. Since $\scl_{\cX}(B)$ is
the smallest ball containing $B$, it must be contained in $B'$ and
therefore, $\scl_{\cX}(B)\cap Y\subseteq B'\cap Y=B$.
\sn
2): \ In view of Lemma~\ref{hullinc}, it suffices to show that
$B_1\ne B_2$ implies $\scl_{\cX}(B_1)\ne\scl_{\cX}(B_2)$. This is a
consequence of part 1) of this lemma.
\sn
3): \ Take a centered system of balls $(B_i)_{i\in I}$ in $(Y,\cB_X\cap
Y)$. Then $(\scl_{\cX}(B_i))_{i\in I}$ is a centered system of balls in
$(X,\cB_X)$ since $B_{i_1}\cap\ldots\cap B_{i_n}\ne\emptyset$ implies that $\scl_{\cX}(B_{i_1})
\cap\ldots\cap\scl_{\cX}(B_{i_n})\ne\emptyset$. By part 1), $B_i=\scl_{\cX}(B_i)\cap Y$, whence
\[
\bigcap_{i\in I} B_i\>=\> \bigcap_{i\in I} \left(\scl_{\cX}(B_i)\cap
Y\right) \>=\>\left(\bigcap_{i\in I} \scl_{\cX}(B_i)\right)\,\cap\,Y\>.
\]
\end{proof}

\pars
With the help of this lemma, we obtain:
\begin{proposition}                         \label{subsc}
Take an {\bf S}$^*$ ball space $(X,\cB_X)$ and assume that $B\cap Y\ne\emptyset$ for every
$B\in\cB_X\,$. Then also $(Y,\cB_X\cap Y)$ is an {\bf S}$^*$ ball space.
\end{proposition}
\begin{proof}
Take a centered system of balls $(B_i)_{i\in\N}$ in $(Y,\cB_X\cap Y)$.
Then by part 3) of Lemma~\ref{SC5sub}, $(\scl_{\cB_X}(B_i))_{i\in\N}$
is a centered system of balls in $(X,\cB_X)$ with $\bigcap_{i\in I} B_i=
\left(\bigcap_{i\in I} \scl_{\cB_X}(B_i)\right)\,\cap\,Y$. Since $(X,\cB_X)$
is assumed to be {\bf S}$^*$, $\bigcap_{i\in I} \scl_{\cB_X}(B_i)$ is a ball
in $\cB_X\,$. Therefore, $\bigcap_{i\in I} B_i= \left(\bigcap_{i\in I}
\scl_{\cB_X}(B_i)\right)\,\cap\,Y\ne\emptyset$ is a ball in $\cB_X\cap Y$.
\end{proof}

%
%
\subsection{Analogues of the Knaster--Tarski Theorem}        \label{sectanKT}
\mbox{ }\n
Proposition~\ref{subsc} can be applied to the special case where a function $f:X\rightarrow X$
is given and  $Y$ is the set $\Fix(f)$ of all fixed points of $f$. Using also
Lemma~\ref{S4cfcl}, we obtain:
\begin{corollary}                             \label{subsccor}
Take an {\bf S}$^*$ ball space $(X,\cB_X)$ and a function $f:X\rightarrow X$.
If each ball in $\cB_X$ contains a fixed point, then
\[
(\Fix(f), \cB_X\cap\Fix(f))
\]
is an {\bf S}$^*$ ball space. If each $f$-closed ball in $\cB_X$ contains a fixed point,
then
\[
(\Fix(f), \cB_X^f\cap\Fix(f))
\]
is an {\bf S}$^*$ ball space.
\end{corollary}

Using these results, a ball spaces analogue of the Knaster-Tarski Theorem can be proved:
\begin{theorem}                             \label{KTball}
Take an {\bf S}$^*$ ball space $(X,\cB)$ and a function $f:X\rightarrow X$.
\sn
1) \ Assume that every ball in $\cB$ contains a fixed point or a smaller ball. Then every ball in
$\cB$ contains a fixed point, and $(\Fix(f), \cB\cap\Fix(f))$ is an {\bf S}$^*$ ball space.
\sn
2) \ Assume that $\cB$ contains an $f$-closed ball and every $f$-closed ball in
$\cB$ contains a fixed point or a smaller $f$-closed ball. Then every $f$-closed ball in $\cB$
contains a fixed point, and $(\Fix(f), \cB^f\cap\Fix(f))$ is an {\bf S}$^*$ ball space.
\end{theorem}
\begin{proof}
1): \ It follows from our assumptions together with Theorem~\ref{basic1c}
that every $B\in\cB$ contains a fixed point. Therefore, $B\cap \Fix(f)\ne\emptyset$.
From Corollary~\ref{subsccor} it follows that $(\Fix(f),\cB^f\cap\Fix(f))$ is an {\bf S}$^*$
ball space.
\sn
2): \ By Lemma~\ref{S4cfcl}, $(X,\cB^f)$ is an {\bf S}$^*$ ball space. Hence it follows from
our assumptions together with part 1) of our theorem, applied to $\cB^f$ in place of $\cB$, that
every $f$-closed ball $B$ in $\cB$ contains a fixed point and that $(\Fix(f),\cB^f\cap\Fix(f))$
is an {\bf S}$^*$ ball space.
\end{proof}

\parm
Let us apply this theorem to the case of topological spaces.
Take a compact topological space $X$ and $(X,\cB)$ the associated
ball space formed by the collection $\cB$ of all nonempty closed sets. If
$f: X\rightarrow X$ is any function, then $\cB^f$ can be taken as the set
of all nonempty closed sets of a (possibly coarser) topology, as arbitrary unions and
intersections of $f$-closed sets are again $f$-closed. By Theorem~\ref{cts} and
Lemma~\ref{S4cfcl}, both $(X,\cB)$ and $(X,\cB^f)$ are {\bf S}$^*$ ball spaces (note that
$\cB^f$ is nonempty since it contains $X$). From part 2) of Theorem~\ref{KTball} we now
obtain the following result:

\begin{theorem}                             \label{topnsf}
Take a compact topological space $X$ and a function $f:X\rightarrow X$.
Assume that every nonempty closed, $f$-closed set contains a fixed point or a
smaller closed, $f$-closed set. Then the topology on the set $\Fix(f)$ of
fixed points of $f$ having $\cB^f\cap\Fix(f)$ as its collection of nonempty closed sets
is itself compact.
\end{theorem}

As we are rather interested in the topology on $\Fix(f)$ induced by the original topology of $X$,
we ask for a criterion on $f$ which guarantees that
\begin{equation}                      \label{eqindbs}
\cB_X^f\cap\Fix(f)\>=\> \cB_X\cap\Fix(f)\>.
\end{equation}
\begin{proposition}                             \label{BcapBfcap}
Take an {\bf S}$^*$ ball space $(X,\cB)$ and a function $f:X\to X$.
\sn
1) If $\cB_X\cap\Fix(f)\ne\emptyset$ and
$B_0\in \cB_X\cap\Fix(f)$ is such that $\scl_{\cB_X}(B_0)$ is $f$-closed, then
\begin{equation}                      \label{eqscl}
\scl_{\cB_X}(B_0)\>=\>\scl_{\cB_X^f}(B_0)\>.
\end{equation}
If this holds for every $B_0\in \cB_X\cap\Fix(f)$, then equation (\ref{eqindbs}) holds.
\sn
2) Assume that $f^{-1}(B)\in\cB_X$ for every
$B\in\cB_X$ that contains a fixed point. Then equation (\ref{eqindbs}) holds.
\end{proposition}
\begin{proof}
1): Pick $B_0\in \cB_X\cap\Fix(f)$. By part 1) of Lemma~\ref{sclsmb}, $\scl_{\cB_X}(B_0)$ is the
smallest of all balls in $\cB_X$ that contain $B_0\,$. Consequently, if $\scl_{\cB_X}(B_0)$ is
$f$-closed, then it is also the smallest of all balls in $\cB_X^f$ that contain $B_0\,$. Then
by part 2) of Lemma~\ref{sclsmb}, it must be equal to $\scl_{\cB_X^f}(B_0)$.

Since $B_0=\scl_{\cB_X}(B_0)\cap\Fix(f)$ by part 1) of Lemma~\ref{SC5sub},
equality (\ref{eqscl}) implies that $B_0=\scl_{\cB_X^f}(B_0)\cap\Fix(f)\in
\cB_X^f\cap\Fix(f)$. If equality (\ref{eqscl}) holds
for all $B_0\in \cB_X\cap\Fix(f)$, then this implies the inclusion ``$\supseteq$'' in
(\ref{eqindbs}). The converse inclusion follows from the fact that $\cB_X^f\subseteq \cB_X\,$.

\sn
2): Pick $B_0\in \cB_X\cap\Fix(f)$. Since $B:=\scl_{\cB_X}(B_0)\in\cB_X\,$, we have by assumption
that $f^{-1}(B)\in\cB_X\,$. All fixed points contained in $B$ are also contained in $f^{-1}(B)$,
hence $B_0\subseteq f^{-1}(B)$. As $B$ is the smallest ball in $\cB_X$ containing $B_0\,$, it
follows that $B\subseteq f^{-1}(B)$ and thus $f(B)\subseteq f(f^{-1}(B))\subseteq B$, i.e., $B$
is $f$-closed. Hence by part 1) of our proposition, (\ref{eqscl}) holds for
arbitrary balls $B_0\in \cB_X\cap\Fix(f)$, which implies that (\ref{eqindbs}) holds.
\end{proof}

\pars
The condition of part 2) of this proposition inspires the following definition.
\begin{definition}                            \label{ballcont}
A function on a ball space $(X,\cB)$ is \bfind{ball continuous} if $f^{-1}(B)\in\cB_X$ for every
$B\in\cB_X$.
\end{definition}

If the function $f$ is continuous in the topology of $X$, then it is ball continuous on the
associated ball space $(X,\cB)$ and the equation (\ref{eqindbs}) follows from
Proposition~\ref{BcapBfcap}. Hence we obtain:
\begin{theorem}                               \label{KTtop}
Take a compact topological space $X$ and a continuous function $f:X\rightarrow X$.
Assume that every nonempty closed, $f$-closed set contains a fixed point or a
smaller closed, $f$-closed set. Then the induced topology on the set $\Fix(f)$ of
fixed points of $f$ is itself compact.
\end{theorem}

\bn
%
%
\section{Set theoretic operations on ball spaces}                  \label{sectuc}
%

%
\subsection{Subsets of ball spaces}
\mbox{ }

\begin{proposition}                      \label{sub}
Take two ball spaces $(X,\cB_1)$ and $(X,\cB_2)$ on the same set $X$ such that $\cB_1\subseteq\cB_2\,$. If
$(X,\cB_2)$ is {\bf S}$_1$ (or {\bf S}$_1^d$ or {\bf S}$_1^c$), then also $(X,\cB_1)$ is {\bf S}$_1$
(or {\bf S}$_1^d$ or {\bf S}$_1^c$, respectively). This does in general not hold for any other property in the
hierarchy.
\end{proposition}
\begin{proof}
The first assertion holds since every nest (or directed system, or centered system) in $\cB_1$ is also a
nest (or directed system, or centered system) in $\cB_2\,$. To prove the second assertion one constructs an
{\bf S}$^*$ ball space $(X,\cB_2)$ and a nest (or directed system, or centered system) $\cN$ such that the
intersection $\bigcap\cN\in\cB_2$ does not lie in $\cN$. Then to obtain $\cB_1$ one removes all balls from $\cB_2$
that lie in $\bigcap\cN$.
\end{proof}

%
%
\subsection{Unions of two ball spaces on the same set} \label{sectsum}
\mbox{ }\sn
The easy proof of the following proposition is left to the reader:
\begin{proposition}                          \label{Bunion}
If $(X,\cB_1)$ and $(X,\cB_2)$ are {\bf S}$_1$ ball spaces on the same set $X$,
then so is $(X,\cB_1\cup\cB_2)$. The same holds with {\bf S}$_2$ or
{\bf S}$_5$ in place of {\bf S}$_1\,$, and for all properties in the hierarchy if $\cB_2$ is finite.
\end{proposition}

Note that the assertion may become false if $\cB_2$ is infinite and we replace {\bf S}$_1$ by {\bf S}$_3$ or
{\bf S}$_4\,$. Indeed, the intersection of a nest in $\cB_1$ may properly contain maximal balls which do not
remain maximal balls contained in the intersection in $\cB_1\cup\cB_2\,$.

It is also clear that in general infinite unions of {\bf S}$_1$ ball spaces on the same set $X$ will not again be
{\bf S}$_1\,$. For instance, ball spaces with just one ball are always {\bf S}$_1\,$, but by a suitable infinite
union of such spaces one can build nests with empty intersection.

\pars
For any ball space $(X,\cB)$, we define the ball space $(X,\widehat\cB)$ by setting:
\[
\widehat\cB\>:=\> \cB\cup\{X\}\>.
\]
Taking $\cB_1=\cB$ and $\cB_2=\{X\}$ in Proposition~\ref{Bunion},
we obtain:
\begin{corollary}                        \label{corBunion}
A ball space $(X,\cB)$ is {\bf S}$_1$ if and only if $(X,\widehat\cB)$ is {\bf S}$_1\,$. The same holds for all
properties in the hierarchy in place of {\bf S}$_1\,$.
\end{corollary}

\mn
%
%
\subsection{Closure under finite unions of balls}       \label{sectun}
\mbox{ }\sn
Take a ball space $(X,\cB)$.
%
%
By $\fun(\cB)$ we denote the set of all unions of finitely many balls in $\cB$.
%
%
The following lemma is inspired by Alexander's Subbase Theorem:
\begin{lemma}                               \label{BfB}
If ${\cal S}$ is a maximal centered system of balls in $\fun(\cB)$ (that
is, no subset of $\fun(\cB)$ properly containg ${\cal S}$ is a centered
system), then there is a subset ${\cal S}_0$ of ${\cal S}$ which is a
centered system in $\cB$ and has the same intersection as ${\cal S}$.
\end{lemma}
\begin{proof}
It suffices to prove the following: if $B_1,\ldots,B_n\in \cB$ such that
$B_1\cup\ldots\cup B_n\in \cS$, then there is some $i\in\{1,\ldots,n\}$
such that $B_i\in\cS$.

Suppose that $B_1,\ldots,B_n\in \cB\setminus \cS$. By the maximality
of $\cS$ this implies that for each $i\in\{1,\ldots,n\}$, $\cS\cup
\{B_i\}$ is not centered. This in turn means that there is a finite
subset $\cS_i$ of $\cS$ such that $\bigcap\cS_i\cap B_i=\emptyset$. But
then $\cS_1\cup\ldots\cup\cS_n$ is a finite subset of $\cS$ such that
\[
\bigcap(\cS_1\cup\ldots\cup\cS_n)\cap(B_1\cup\ldots\cup B_n)
\>=\>\emptyset\>.
\]
This yields that $B_1\cup\ldots\cup B_n\notin \cS$, which proves our
assertion.
\end{proof}
\pars
The centered systems of balls in a ball space form a poset under inclusion. Since the union of every chain of
centered systems is again a centered system, this poset is chain complete. Hence by Corollary~\ref{ZLcor}
every centered system is contained in a maximal centered system. We use this to prove:
\begin{theorem}                        \label{thfun}
If $(X,\cB)$ is an {\bf S}$_1^c$ ball space, then so is $(X,\fun(\cB))$.
\end{theorem}
\begin{proof}
Take a centered system $\cS'$ of balls in $\fun(\cB)$. Take a maximal centered
system $\cS$ in $\fun(\cB)$ which contains $\cS'$. By Lemma~\ref{BfB} there is a
centered system $\cS_0$ of balls in $\cB$ such that $\bigcap\cS_0=
\bigcap\cS\subseteq\bigcap\cS'$. Since $(X,\cB)$ is an {\bf S}$_1^c$ ball
space, we have that $\bigcap\cS_0\ne\emptyset$, which yields that
$\bigcap\cS'\ne\emptyset$. This proves that $(X,\fun(\cB))$
is an {\bf S}$_1^c$ ball space.
\end{proof}
\sn
In \cite{BKK} it is shown that the theorem becomes false if ``{\bf S}$_1^c$'' is replaced by ``{\bf S}$_1$''.

\parm
In \cite{BKK}, the notion of ``hybrid ball space'' is introduced. The idea is to start with the
union of two ball spaces as in Section~\ref{sectsum} and then close under finite unions. The
question is whether the resulting ball space is an {\bf S}$_1$ ball space if the original ball
spaces are. On symmetrically complete ordered
fields $K$ we have two {\bf S}$_1$ ball spaces: $(K,\cBcb)$ and $(K,\cB_u)$ where $u$ is the
ultrametric induced by the natural valuation of $(K,<)$ (cf.\ Theorem~\ref{symcompsphcomp}). But
by Proposition~\ref{limitedS}, $(K,\cBcb)$ is not {\bf S}$_1^c$, hence Theorem~\ref{thfun} cannot
be applied. Nevertheless, the following result is proven in \cite{BKK} by a direct proof. The
principles that make it work still remain to be investigated more closely.
\begin{theorem}
Take a symmetrically complete ordered field $K$ and $\cB$ to be the set
of all convex sets in $K$ that are finite unions of closed
intervals and ultrametric balls. Then $(K,\cB)$ is spherically complete.
\end{theorem}

\mn
%
%
\subsection{Closure under nonempty intersections of balls}       \label{sectint}
\mbox{ }\sn
Take a ball space $(X,\cB)$. We define:
\mn
(a) \ $\int(\cB)$ to be the set of all nonempty intersections of arbitrarily many balls in $\cB$,
\n
(b) \ $\fint(\cB)$ to be the set of all nonempty intersections of finitely many balls in $\cB$,
\n
(c) \ $\nint(\cB)$ to be the set of all nonempty intersections of nests in $\cB$.
\sn
Note that $(X,\cB)$ is intersection closed if and only if $\int(\cB)=\cB$, finitely intersection closed if and only
if $\fint(\cB)=\cB$, and chain intersection closed if and only if $\nint(\cB)=\cB$. If $(X,\cB)$ is {\bf S}$_5\,$,
then $\nint(\cB)=\cB$.
%
%
If $(X,\cB)$ is {\bf S}$^*\,$, then $\int(\cB)=\cB$ by Proposition~\ref{S*ic}.
We note:
\begin{proposition}
Take an arbitrary ball space $(X,\cB)$. Then the ball space $(X,\int(\cB))$ is intersection closed, and
$(X,\fint(\cB))$ is finitely intersection closed.
\end{proposition}
\begin{proof}
Take balls $B_i\in\int(\cB)$, $i\in I$, and for every $i\in I$, balls $B_{i,j}\in\cB$, $j\in
J_i\,$, such that $B_i=\bigcap_{j\in J_i} B_{i,j}\,$. Then
\[
\bigcap_{i\in I} B_i\>=\> \bigcap_{i\in I,\,j\in J_i} B_{i,j}\>\in\,\int(\cB)\>.
\]
If $I$ is finite and $B_i\in\fint(\cB)$ for every $i\in I$, then every $J_i$ can be taken to be
finite and thus the right hand side is a ball in $\fint(\cB)$.
\end{proof}
\sn
In view of these facts, we introduce the following notions.
\begin{definition}
We call $(X, \int(\cB))$ the \bfind{intersection closure} of $(X,\cB)$, and $(X, \fint(\cB))$
the \bfind{finite intersection closure} of $(X,\cB)$. If a chain intersection closed ball space
$(X,\cB')$ is obtained from $(X,\cB)$ by a (possibly transfinite) iteration of the process of
replacing $\cB$ by $\nint(\cB)$, then we call $(X,\cB')$ a \bfind{chain intersection closure}
of $(X,\cB)$.
\end{definition}

Chain intersection closures are studied in \cite{KKub} and conditions are given for
$(X,\nint(\cB))$ to be the
chain intersection closure of $(X,\cB)$. As stated already in part 1) of Theorem~\ref{ultcic},
this holds for classical ultrametric spaces. This result follows from a
more general theorem (cf.\ \cite[Theorem~2.2]{KKub}):
\begin{theorem}
If $(X,\cB)$ is a tree-like ball space, then $(X,\nint(\cB))$ is its chain intersection closure,
and if in addition $(X,\cB)$ is an {\bf S}$_1$ ball space, then so is $(X,\nint(\cB))$.
\end{theorem}

Since chain intersection closed {\bf S}$_1$ ball spaces are {\bf S}$_5\,$, we obtain:
\begin{corollary}
If $(X,\cB)$ is a tree-like {\bf S}$_1$ ball space, then $(X,\nint(\cB))$ is an {\bf S}$_5$ ball
space.
\end{corollary}

Intersection closure can also increase the strength of ball spaces:
\begin{theorem}                                    \label{thint}
If $(X,\cB)$ is an {\bf S}$_1^c$ ball space, then its intersection closure $(X,\int(\cB))$ is
an {\bf S}$^*$ ball space.
\end{theorem}
\begin{proof}
Take a centered system $\{B_i\mid i\in I\}$ in $(X,\int(\cB))$. Write $B_i=\bigcap_{j\in J_i} B_{i,j}$ with $B_{i,j}
\in \cB$. Then $\{B_{i,j}\mid i\in I,j\in J_i\}$ is a centered system in $(X,\cB)$: the intersection of finitely
many balls $B_{i_1,j_1},\ldots,B_{i_n,j_n}$ contains the intersection $B_{i_1}\cap\ldots\cap B_{i_n}\,$, which by
assumption is non\-empty. Since $(X,\cB)$ is {\bf S}$_1^c$, $\bigcap_i B_i =\bigcap_{i,j} B_{i,j}\ne
\emptyset$. This proves that $(X,\int(\cB))$ is an {\bf S}$_1^c$ ball space. Since $(X,\int(\cB))$ is intersection
closed, Theorem~\ref{sccent} now shows that $(X,\int(\cB))$ is an {\bf S}$^*$ ball space.
\end{proof}

%

\mn
%
%
\subsection{Closure under finite unions and under intersections}       \label{sectunint}
\mbox{ }\sn
From Theorems~\ref{thfun} and~\ref{thint} we obtain:
\begin{theorem}                                                            \label{thfunint}
Take any ball space $(X,{\cB})$. If $\cB'$ is obtained from $\cB$ by first closing under finite unions
and then under arbitrary nonempty intersections, then:
\sn
1) \ $\cB'$ is closed under finite unions,
\sn
2) \ $\cB'$ is intersection closed,
\sn
3) \ if in addition $(X,{\cB})$ is an {\bf S}$_1^c$ ball space, then $(X,\cB')$ is an {\bf S}$^*$ ball space.
\end{theorem}
\begin{proof}
1): \ Take $S_1,\ldots,S_n\subseteq\fun(\cB)$
such that $\bigcap S_i\ne\emptyset$ for $1\leq i\leq n$. Then
\[
\left(\bigcap S_1\right)\cup\ldots\cup \left(\bigcap S_n\right) \>=\> \bigcap \{B_1\cup\ldots\cup B_n\mid B_i\in S_i
\mbox{ for } 1\leq i\leq n\}\>.
\]
Since $B_i\in \fun(\cB)$ for $1\leq i\leq n$, we have that also $B_1\cup\ldots\cup B_n\in \fun(\cB)$. This implies
that $(\bigcap S_1)\cup\ldots\cup (\bigcap S_n)\in\cB'$.
\sn
2): \ Since $\cB'$ is an intersection closure, it is intersection closed.
\sn
3): \  By Theorems~\ref{thfun} and~\ref{thint}, $(X,\cB')$ is an {\bf S}$^*$ ball space.
\end{proof}

\mn
%
%
\subsection{The topology associated with a ball space}       \label{sectasstop}
\mbox{ }\sn
Take any ball space $(X,\cB)$. Theorem~\ref{thfunint} tells us that in a canonical way we can associate
with it a ball space $(X,\cB')$ which is closed under nonempty intersections and under finite unions.
If we also add $X$ and $\emptyset$ to $\cB'$, then we obtain the collection of closed sets for a topology
whose associated ball space is $(X,\cB'\cup\{X\})$.
%
%
\begin{theorem}                                             \label{asstopco}
The topology associated with a ball space $(X,{\cB})$ is compact if and only if $(X,{\cB})$ is an {\bf S}$_1^c$
ball space.
\end{theorem}
\begin{proof}
The ``if'' direction of the equivalence follows from Theorems~\ref{thfunint} and~\ref{cts}. The other direction
follows from Theorem~\ref{cts} and Proposition~\ref{sub}.
\end{proof}

\mn
{\bf Example: the $p$-adics.}
\sn
The field ${\mathbb Q}_p$ of $p$-adic numbers together with the $p$-adic valuation $v_p$ is
spherically complete.
(This fact can be used to prove the original Hensel's Lemma via the ultrametric fixed point theorem,
see \cite{[PC]}, or even better, via the ultrametric attractor theorem, see \cite{[KU3]}.) The associated ball
space is a classical ultrametric ball space and hence tree-like. It follows from Proposition~\ref{thbsut}
that it is an {\bf S}$_1^c$ ball space. Hence by Theorem~\ref{asstopco} the topology derived from this ball space
is compact.

However, ${\mathbb Q}_p$ is known to be locally compact, but not compact under the topology induced by the $p$-adic
metric. But in this topology the ultrametric balls $B_\alpha (x)$ are basic open sets, whereas in the topology
derived from the ultrametric ball space they are closed and their complements are the basic open sets. It follows
that the balls $B_\alpha (x)$ are not open. It thus turns out that the usual $p$-adic topology on ${\mathbb Q}_p$
is strictly finer than the one we derived from the ultrametric ball space.

\bn
%
%
\section{Tychonoff type theorems}      \label{sectTy}\label{sectprod}
%
%
%
\subsection{Products in ball spaces}
\mbox{ }\sn
In \cite{BKK} it is shown that the category consisting of all ball spaces together with the ball
continuous functions (see Definition~\ref{ballcont}) as morphisms allows products and coproducts.
The products can be defined as follows.

Assume that $(X_j,\cB_j)_{j\in J}$ is a family of ball spaces. Recall that
$\widehat\cB_j=\cB_j\cup\{X_j\}$.
\begin{definition}
We set $X=\prod_{j\in J} X_j$ and define the
\bfind{product} $(X_j,\cB_j)_{j\in J}\pr$ to be $(X,(\cB_j)_{j\in J}\pr)$, where
\[
(\cB_j)_{j\in J}\pr \>:=\> \left\{\prod_{j\in J} B_j\mid \mbox{ for some }k\in J,\> B_k\in\cB_k\mbox{ and }
\forall j\ne k: B_j=X_j \right\}\>.
\]
Further, we define the \bfind{topological product} $(X_j,\cB_j)_{j\in J}\tpr$ to be $(X,(\cB_j)_{j\in J}\tpr)$,
where
\[
(\cB_j)_{j\in J}\tpr := \left\{\prod_{j\in J} B_j\mid \forall j\in J:\>B_j\in
\widehat\cB_j\mbox{ and } B_j=X_j \mbox{ for almost all } j\right\}\>,
\]
and the \bfind{box product} $(X_j,\cB_j)_{j\in J}\bpr$ of the family to be $(X,(\cB_j)_{j\in J}\bpr)$, where
\[
(\cB_j)_{j\in J}\bpr \>:=\>\left\{\prod_{j\in J} B_j\mid \forall j\in J:\>B_j\in\cB_j\right\}\>.
\]
Since the sets $\cB_i$ are nonempty, it follows that $\cB\ne\emptyset$, and as no ball in any $\cB_i$ is empty,
it follows that no ball in $(\cB_j)_{j\in J}\pr$, $(\cB_j)_{j\in J}\tpr$ and $(\cB_j)_{j\in J}\bpr$ is empty.
\end{definition}

\pars
We leave the proof of the following observations to the reader:
\begin{proposition}                \label{propproducts}
\n
a) We have that
\[
(\cB_j)_{j\in J}\pr\>\subseteq\> (\cB_j)_{j\in J}\tpr\>=\> (\widehat\cB_j)_{j\in J}\tpr \>\subseteq\>
(\widehat\cB_j)_{j\in J}\bpr\>.
\]
b) The following equations hold:
\begin{eqnarray*}
\fint\left((\widehat\cB_j)_{j\in J}\pr\right) &=& \fint\left((\cB_j)_{j\in J}\tpr\right)
\>=\> (\fint(\cB_j))_{j\in J}\tpr\>, \\
\int\left((\widehat\cB_j)_{j\in J}\pr\right) &=& \int\left((\cB_j)_{j\in J}\tpr\right)
\>=\> (\int(\widehat\cB_j))_{j\in J}\bpr\>.
\end{eqnarray*}
\end{proposition}

\parm
The following theorem presents our main results on the various products.
\begin{theorem}                                   \label{bprtpr}
The following assertions are equivalent:
\sn
a) \ the ball spaces $(X_j,{\cal B}_j)$, $j\in J$, are spherically complete,
\sn
b) \ their box product is spherically complete,
\sn
c) \ their topological product is spherically complete.
\sn
d) \ their product is spherically complete.
\sn
The same holds with  ``\/{\bf S}$_1^d$'' and ``\/{\bf S}$_1^c$'' in place of ``spherically complete''.

\parm
The equivalence of a) and b) also holds for all other properties in the hierarchy,
and the equivalence of a) and d) also holds for {\bf S}$_2\,$, {\bf S}$_3\,$, {\bf S}$_4$ and {\bf S}$_5\,$.
\end{theorem}
\begin{proof}
Take ball spaces $(X_j,\cB_j)$, $j\in J$, and in every $\cB_j$ take a set of balls $\{B_{i,j}\mid i\in I\}$.
Then we have:
\begin{equation}                                  \label{pr}
\bigcap_{i\in I}^{} \prod_{j\in J} B_{i,j}
\>=\> \prod_{j\in J} \bigcap_{i\in I}^{} B_{i,j} \>.
\end{equation}
If $\cN=(\prod_{j\in J} B_{i,j})_{i\in I}$ is a nest of balls in $(\prod_{j\in J} X_j,
(\cB_j)_{j\in J}\bpr)$, then for every $j\in J$, also $(B_{i,j})_{i\in I}$ must be a nest in $(X_j,\cB_j)$.

\sn
a) $\Rightarrow$ b):
Assume that all ball spaces $(X_j,{\cal B}_j)$, $j\in J$, are spherically complete.
Then for every $j\in J$,  $(B_{i,j})_{i\in I}$ has nonempty intersection. By (\ref{pr})
it follows that $\bigcap {\cal N}\ne\emptyset$. This proves the implication a) $\Rightarrow$ b).

\sn
b) $\Rightarrow$ a):
Assume that $(\prod_{j\in J} X_j,(\cB_j)_{j\in J}\bpr)$ is spherically
complete. Take $j_0\in J$ and a nest of balls ${\cal N}=
(B_i)_{i\in I}$ in $(X_{j_0},{\cal B}_{j_0})$. For each $i\in I$, set
$B_{i,j_0}=B_i$ and $B_{i,j}=B_{0,j}$ for $j\ne j_0$ where $B_{0,j}$ is an arbitrary fixed ball in
$\cB_j\,$. Then $(\prod_{j\in J} B_{i,j})_{i\in I}$
is a nest in $(\prod_{j\in J} X_j,(\cB_j)_{j\in J}\bpr)$. By assumption,
\[
\emptyset\>\ne\> \bigcap_{i\in I}^{} \prod_{j\in J} B_{i,j}\>=\>
\left(\bigcap_{i\in I}^{} B_i\right)\times \left(\prod_{j_0\ne j\in J} B_{0,j} \right)\>,
\]
whence $\bigcap_{i\in I}^{} B_i\ne\emptyset$. We have shown that for
every $j\in J$, $(X_j,{\cal B}_j)$ is spherically complete. This proves the implication b) $\Rightarrow$ a).

\sn
a) $\Rightarrow$ c):
Assume that all ball spaces $(X_j,{\cal B}_j)$, $j\in J$, are spherically complete. Then by
Corollary~\ref{corBunion}, all ball spaces $(X_j,\widehat\cB_j)$, $j\in J$, are spherically complete. By the
already proven implication a) $\Rightarrow$ b), their box product $(X_j,\widehat\cB_j)_{j\in J}\bpr$ is
spherically complete. By part a) of Proposition~\ref{propproducts} together with Proposition~\ref{sub},
$(X,\cB)_{j\in J}\tpr$ is spherically complete, too.

\sn
c) $\Rightarrow$ d): Again, by part a) of Proposition~\ref{propproducts} together with Proposition~\ref{sub},
the product of the ball spaces $(X_j,\widehat\cB_j)$, $j\in J$, is spherically complete, and as the product of
the ball spaces $(X_j,\cB_j)$, $j\in J$, is a subspace of this, it is also spherically complete.
\sn
d) $\Rightarrow$ a):
Same as the proof of b) $\Rightarrow$ a), where we now take $B_{0,j}=X_j\,$.

\parm
These proofs also work when ``spherically complete'' is replaced by ``\/{\bf S}$_1^d$'' or ``\/{\bf S}$_1^c$'',
as can be deduced from the following observations:
\sn
1) $\{\prod_{j\in J} B_{i,j} \mid i\in I\}$ is a
centered system if and only if all sets $\{B_{i,j}\mid i\in I\}$, $j\in J$, are.
\sn
2) If $\{\prod_{j\in J} B_{i,j} \mid i\in I\}$ is a directed system, then so are $\{B_{i,j}\mid i\in I\}$ for all
$j\in J$.
\sn
3) Fix $j_0\in J$. If $\{B_{i,j_0}\mid i\in I\}$ is a directed system, then so is $\{\prod_{j\in J} B_{i,j} \mid
i\in I\}$ when the balls are chosen as in the proof of b) $\Rightarrow$ a) or d) $\Rightarrow$ a).

\parm
A proof of the equivalence of a) and b) similar to the above also holds for all other properties in the hierarchy.
For the properties {\bf S}$_2\,$, {\bf S}$_3$, {\bf S}$_4$ and {\bf S}$_5\,$, one uses the fact that by definition,
$\prod_{j\in J} B_j$ is a ball in $(\cB_j)_{j\in J}\bpr$ if and only if every $B_j$  is a ball in $\cB_j$ and that
\sn
4) $\prod_{j\in J} B'_j$ is a ball contained in $\prod_{j\in J} B_j$ if and only if every $B'_j$ is a ball
contained in $B_j\,$,
\sn
5) $\prod_{j\in J} B'_j$ is a maximal (or largest) ball contained in $\prod_{j\in J} B_j$ if and only if every
$B'_j$ is a maximal (or largest, respectively) ball contained in $B_j\,$.
\end{proof}

\begin{example}
There are {\bf S}$^*$ ball spaces $(X_j,\cB_j)$, $j\in\N$, such that the ball space $(X,(\cB_j)_{j\in\N}\tpr)$ is
not even {\bf S}$_2\,$.
\rm
Indeed, we choose a set $Y$ with at least two elements, and for every $j\in\N$ we take $X_j=Y$ and $\cB_j=\{B\}$
with $\emptyset\ne B\ne Y$. Then trivially, all ball spaces $(X_j,\cB_j)$ are {\bf S}$^*$. For all $i,j\in\N$,
define
\[
B_i\>:=\>
\underbrace{B\times B\times\ldots\times B}_{i\text{ times }}\times Y\times Y\times\ldots\>\in(\cB_j)_{j\in\N}\tpr.
\]
Then $\cN=\{B_i\mid i\in I\}$ is a nest of balls in $(\cB_j)_{j\in\N}\tpr$, but the intersection
$\bigcap\cN=\prod_{j\in\N}B$ does not contain any ball in this ball space.
\end{example}

\begin{example}
There are {\bf S}$^*$ ball spaces $(X,\cB_j)$, $j=1,2$, such that the ball space $(X,(\cB_j)_{j\in\{1,2\}}\pr)$
is not {\bf S}$_2^c$.
\rm
Indeed, we choose again a set $Y$ with at least two elements and take $\cB_1=\cB_2=\{B\}$ with $\emptyset\ne B
\ne Y$. Then as in the previous example, $(X_j,\cB_j)$, $j=1,2$
are {\bf S}$^*$ ball spaces. Further, $(\cB_j)_{j\in\{1,2\}}\pr=\{Y\times Y,B\times Y,Y\times B\}$, which
is a centered system whose intersection does not contain any ball.
\end{example}

\mn
%
%
\subsection{The ultrametric case}
\mbox{ }\sn
If $(X_j,u_j)$, $j\in J$ are ultrametric spaces with value sets $u_j X_j=\{u_j(a,b)\mid a,b\in X_j\}$, and if
$B_j=B_{\gamma_j}(a_j)$ is an ultrametric ball in $(X_j,u_j)$ for each $j$, then
\[
\prod_{j\in J} B_j\>=\>\{(b_j)_{j\in J}\mid \forall j\in J:\, u_j(a_j,b_j)\leq \gamma_j\}\>.
\]
This shows that the box product is the ultrametric ball space for the product ultrametric on $\prod_{j\in J} X_j$
which is defined as
\[
u_{\rm prod}((a_j)_{j\in J}\,,\,(b_j)_{j\in J})\>=\> (u_j(a_j,b_j))_{j\in J}\>\in \prod_{j\in J} u_j X_j\>.
\]
The latter is a poset, but in general not totally ordered, even if all $u_j X_j$ are totally ordered
and even if $J$ is finite.
So the product ultrametric is a natural example for an ultrametric with partially ordered value set.

If the index set $J$ is finite and all $u_j X_j$ are contained in some totally ordered set $\Gamma$ such that all
of them have a common least element $0\in\Gamma$, then we can define an ultrametric on the product
$\prod_{j\in J} X_j$ which takes values in $\bigcup_{j\in J} u_j X_j\subseteq\Gamma$ as follows:
\[
u_{\rm max}((a_j)_{j\in J}\,,(b_j)_{j\in J})\>=\> \max_j u_j(a_j,b_j)
\]
for all $(a_j)_{j\in J},(b_j)_{j\in J}\in\prod_{j\in J} X_j\,$. We leave it to the reader to prove that this is
indeed an ultrametric. The corresponding ultrametric balls are the sets of the form
\[
\{(b_j)_{j\in J}\mid \forall j\in J:\, u_j(a_j,b_j)\leq \gamma\}
\]
for some $(a_j)_{j\in J}\in\prod_{j\in J} X_j$ and $\gamma\in\bigcup_{j\in J} u_j X_j\,$. Now the value set is
totally ordered. It turns out that the collection
of balls so obtained is a (usually proper) subset of the full ultrametric ball space of the product ultrametric.
Therefore, if all $(X_j,u_j)$ are spherically complete, then so is $(\prod_{j\in J} X_j,u_{\rm max})$
by Theorem~\ref{bprtpr} and Proposition~\ref{sub}.

\pars
\begin{theorem}                                   
Take ultrametric spaces $(X_j,u_j)$, $j\in J$. Then the ultrametric space $(\prod_{j\in J} X_j,u_{\rm prod})$ is
spherically complete if and only if all $(X_j,u_j)$, $j\in J$, are spherically complete.

If the index set $J$ is finite and all $u_j X_j$ are contained in some totally ordered set $\Gamma$ such that all
of them have a common least element, then the same also holds for $u_{\rm max}$ in place of $u_{\rm prod}\,$.
\end{theorem}

\begin{proof}
As was remarked earlier, the ultrametric ball space of the product ultrametric is the box product of the
ultrametric ball spaces of the ultrametric spaces $(X_j,u_j)$. Thus the first part of the theorem is a corollary to
Theorem~\ref{bprtpr}.

To prove the second part of the theorem, it suffices to prove the converse of the implication we have stated just
before the theorem. Assume that the space $(\prod_{j\in J}X_j,u_{\rm max})$ is spherically complete and choose
any $j_0\in J$. Let $\cN_{j_0}=\{B_{\gamma_i}(a_{i,j_0})\mid i\in I\}$ be a nest of balls in $(X_{j_0},u_{j_0})$.
Further, for every $j\in J\setminus\{j_0\}$ choose some element $a_j\in X_j$ and for every
$i\in I$ set $a_{i,j}:=a_j$ and
\begin{eqnarray*}
B_i&:=& \{(b_j)_{j\in J}\in\prod_{j\in J} X_j\mid u_{\rm max} ((a_{i,j})_{j\in J},(b_j)_{j\in J})\le\gamma_i\}\\
&=& \{(b_j)_{j\in J}\in\prod_{j\in J} X_j\mid\forall j\in J:\ u_j(a_{i,j},b_j)\le\gamma_i\}\>.
\end{eqnarray*}
In order to show that $\cN:=\{B_i\mid i\in I\}$ is a nest of balls in $(\prod_{j\in J}X_j,u_{\rm max})$, we have to
show that any two balls $B_i\,$, $B_k\,$, $i,k\in I$, have nonempty intersection. Assume without loss of generality
that $\gamma_i\leq\gamma_k\,$. As $\{B_{\gamma_i}(a_{i,j_0})\mid i\in I\}$ is a nest of balls, we have that
$a_{i,j_0}\in B_{\gamma_k}(a_{k,j_0})$. It follows that $u_{j_0}(a_{k,j_0},a_{i,j_0})\le\gamma_k\,$, and since $a_{i,j}=a_j=a_{k,j}$ for every $j\in J\setminus\{j_0\}$,
\[
(a_{i,j})_{j\in J} \>\in\> B_i\>\cap\> \{(b_j)_{j\in J}\in\prod_{j\in J} X_j\mid\forall j\in J:\
u_j(a_{k,j},b_j)\le\gamma_k\} \>=\> B_i\cap B_k\>.
\]
As $(\prod_{j\in J}X_j,u_{\rm max})$ is assumed to be spherically complete, there is some $(z_j)_{j\in J}
\in\bigcap\cN$; it satisfies $u_j(a_{i,j},z_j)\le\gamma_i$ for all $i\in I$ and all $j\in J$. In particular, taking
$j=j_0\,$, we find that $z_{j_0}\in B_{\gamma_i}(a_{i,j_0})$ for all $i\in I$ and thus, $z_{j_0}\in\bigcap
\cN_{j_0}$.
\end{proof}

\mn
%
%
\subsection{The topological case}
\mbox{ }\sn
In which way does Tychonoff's theorem follow from its analogue for ball spaces? The problem in
the case of topological spaces is that the topological product ball space we have defined, while
containing only closed sets of the product, does not contain all of them, as it is not
necessarily closed under finite unions and arbitrary
intersections. We have to close it under these operations.

If the topological spaces $X_i\,$, $i\in I$, are compact, then their associated ball spaces
$(X_i,\cB_i)$ are {\bf S}$_1^c$ (cf.\ Theorem~\ref{cts}). By Theorem~\ref{bprtpr} their
topological product is also {\bf S}$_1^c$.
Theorem~\ref{thfunint} shows that the product topology of the topological spaces $X_i$ is the
closure of $(\cB_j)_{j\in J}\tpr$ under finite unions and under arbitrary nonempty
intersections, when $\emptyset$ and the whole space are adjoined. By Theorem~\ref{asstopco},
this topology is compact.

We have shown that Tychonoff's Theorem follows from its ball spaces analogue.

\end{document}